\documentclass{article}
\usepackage{amsmath, amsthm, amsfonts, amssymb}
\usepackage{hyperref}
\usepackage{wasysym}
\newtheorem{theorem}{Theorem}[section]

\theoremstyle{definition}
\newtheorem{definition}[theorem]{Definition}
\newtheorem{example}[theorem]{Example}

\theoremstyle{remark}
\newtheorem{remark}[theorem]{Remark}
\numberwithin{equation}{section}

\newcommand{\oa}{\left\{}
\newcommand{\fa}{\right\}}

\renewcommand{\exp}[1]{\textrm{exp}\oa #1 \fa}

\begin{document}
\title{\bf Hyperbolic valued random variables and conditional expectation}
\date{\textbf{Romesh Kumar and Kailash Sharma}}
\vspace{0in}
\maketitle

$\textbf{Abstract.}$ In this paper, we introduce the concept of hyperbolic valued random variables, their expectation and moments.We develop the hyperbolic analogue of Binomial and Poisson distributions. We study some of the properties of expectation on  the basis of decomposition of a hyperbolic number into idempotent components. Finally we define conditional expectation of hyperbolic valued random variables and study some of its basic properties. Our random variable can take values which are zero divisors and this is the important part of this study. \\\\
 $\textbf{Keywords.}$  $\mathbb{D}$-probabilistic space, hyperbolic valued random variables, expectation, moments,  moment generating function, Binomial distribution, conditional expectation, partition.

\begin{section} {Introduction}
The hyperbolic numbers in the mathematical literature has been called with different names: split-complex numbers, double numbers, perplex numbers and duplex numbers. The set $\mathbb{D}$ of hyperbolic numbers has a partial order relation $\preceq $ which is one of the interesting property of the set $\mathbb{D}.$ We can think of random variables taking values in the set of hyperbolic numbers being motivated by the work of D. Alpay, M.E. Luna-Elizarraras and M. Shapiro [3].\\
We start with a $\mathbb{D}$-probabilistic space $(\Omega, \Sigma, P_{\mathbb{D}}),$ where $ \Omega$ is the set of all possible outcomes of the experiment, $\Sigma $ is the set of all events and $ P_{\mathbb{D}} $ is the on  $\mathbb{D}$-valued probabilistic measure on $\Sigma $ measuring the $\mathbb{D}$-probabilities of events. A $\mathbb{D}$- random variable is a function $$ X_{\mathbb{D}}\; \colon \Omega \longrightarrow \mathbb{D} $$ such that $$ X_{\mathbb{D}}^{-1}(D)\in \Sigma $$ for every open set D  in $\mathbb{D}.$ The expectation of $ X_{\mathbb{D}}$ is defined as $$E(X_{\mathbb{D}})=\int_{\Omega}X_{\mathbb{D}}\; dP_{\mathbb{D}},$$ provided the integral exists. We introduce the conditional expectation of $ X_{\mathbb{D}}$ given an event B with $ P_{\mathbb{D}}(B)$ taking the values which are zero divisors or not zero divisors. The values taken by $ P_{\mathbb{D}}(B)$ which are zero divisors are of special interest in this work. This work can be seen as a continuation of work of Daniel Alpay et.al. \cite {DA Kolax}. It seems that the whole probability theory can be generalized in this direction.This work may have important applications in mathematical statistics, thermodynamics and statistical physics see \cite {DA Kolax}. 
\end{section}

\begin{section} {A review of hyperbolic numbers}
The ring of hyperbolic numbers is the commutative ring $\mathbb{D}$ defined as $$\mathbb{D}= \left\{a+bk\;|\;k^{2}=1,k\notin\mathbb{R};a,b\in\mathbb{R}\right\}.$$  The $ \dagger $-conjugation of a hyperbolic number $ z=a+bk$ is given by $$ z^{\dagger}= a-bk$$ which is additive, involutive and multiplicative operation on $\mathbb{D}$.\;Note that given  $ z= a+bk\in\mathbb{D},$ then $zz^{\dagger}=a^{2}-b^{2}\in \mathbb{R},$ from which it follows that any hyperbolic number z with $zz^{\dagger}\neq0$ is invertible, and its inverse is given by $$ z^{-1}=\dfrac{z^{\dagger}}{zz^{\dagger}}.$$If, on the other hand, $z\neq0$ but $zz^{\dagger}=a^{2}-b^{2}=0$, then z is a zero divisor. In fact there are no other zero divisors.Thus the set of zero divisors, denoted by $\mathfrak{S}_{\mathbb{D}}$, is
\begin{equation*}\label{xx}
\begin{split}
\mathfrak{S}_{\mathbb{D}}& = \left\{ z= a+bk \; | \; z \neq 0, \;  \; zz^ {\dagger} = a^{2}-b^{2}=0 \right\}   \\
& = \left\{ z = a( 1\pm k)\; |\; a \neq 0 \in \mathbb{R}  \right\}\\
\end{split} 
\end{equation*}
There are two very special zero divisors in $ \mathbb{D}$ which are $$e=\frac{1+k}{2}$$ and its   $ \dagger $-conjugation\;$$ e^{\dagger}=\frac{1-k}{2}.$$
These are mutually complementary idempotent elements in $\mathbb{D}$.
The two sets $\mathbb{D}_{e} =e\;\mathbb{D}$ and $\mathbb{D}_{e^{\dagger}} = e^{\dagger}\;\mathbb{D}$ are principal ideals in the ring $\mathbb{D}$ such that $\mathbb{D}_{e}\cap \mathbb{D}_{e^{\dagger}} =\left\lbrace0\right\rbrace$ and $\mathbb{D}= \mathbb{D}_{e}+ \mathbb{D}_{e^{\dagger}}, $ which is idempotent decomposition of  $ \mathbb{D}.$ Every hyperbolic number $z=a+bk$ can be written as $z=a+bk=(a+b)e+(a-b)e^{\dagger}=v_{1}e +v_{2}e^{\dagger},$ which is the idempotent decomposition of a hyperbolic number.
The algebraic operations of addition, multiplication, taking of inverse, etc. can be performed component-wise.
Observe that the sets $\mathbb{D}_{e}$ and $\mathbb{D}_{e^{\dagger}}$ can be written as $\mathbb{D}_{e}$ =$\left\lbrace r\;e\;|\; r \in \mathbb{R}\right\rbrace$ =$ \mathbb{R}\;e$ and $\mathbb{D}_{e^{\dagger}}$ =$\left\lbrace t\;e^{\dagger}\;|\; t \in \mathbb{R}\right\rbrace$ =$ \mathbb{R}\;e^{\dagger}.$
The set $\mathbb{D}$ of hyperbolic numbers is a vector space over the field $\mathbb{R}$ of real numbers with basis$ \left\lbrace e,\;e^{\dagger}\right\rbrace$ which is isomorphic to the linear space of complex numbers over the field $ \mathbb{R}$ of real numbers.
The set of non negative hyperbolic numbers is $$\mathbb{D}^{+} =\left\lbrace v_{1}\;e +v_{2}\;e^{\dagger}\;|\;v_{1},v_{2}\geq0\right\rbrace.$$ \\
We will need two more sets \\
$$\mathbb{D}_{e}^{+}= \mathbb{D}_{e}\cap \mathbb{D}^{+}-\left\lbrace 0 \right\rbrace $$ and $$\mathbb{D}_{e^{\dagger}}^{+} = \mathbb{D}_{e^{\dagger}}\cap \mathbb{D}^{+} - \left\lbrace 0 \right\rbrace.$$\\
Given $ z_{1},z_{2}\in \mathbb{D}$, we write $z_{1}\preceq z_{2}\; whenever\; z_{2}-z_{1}\; \in \mathbb{D}^{+}.$This relation is a partial order relation in $\mathbb{D}$ which is an extension of total order relation $ \leq $ on $\mathbb{R}.$\\
Given any hyperbolic number $ \alpha, $ one can see that the entire hyperbolic plane is divided into four quarters: the quarter plane of hyperbolic numbers which are $\mathbb{D}$- less than or equal to $ \alpha; $ the quarter plane of hyperbolic numbers which are $\mathbb{D}$- greater than $ \alpha; $ and the two quarter planes where the hyperbolic numbers are not $\mathbb{D}$- comparable with $ \alpha. $ Let us denote by $ A_{\alpha}, $ the set of all hyperbolic numbers which are not $\mathbb{D}$- comparable with $ \alpha. $ Then $$ \mathbb{D} = \left\lbrace {z \in \mathbb{D}\; | \; z \preceq \alpha} \right\rbrace \cup \left\lbrace {z \in \mathbb{D}\; | \; z \succ\alpha } \right\rbrace \cup A_{\alpha}. $$    
The hyperbolic valued modulus on $\mathbb{D}$ is defined by $$ |z|_{k} = | v_{1}e +v_{2}e^{\dagger}|_{k} = |v_{1}|e + |v_{2}|e^{\dagger} \in \mathbb{D}^{+},$$ where $|v_{1}|,|v_{2}|$ denote the usual modulus of real numbers. The set $\mathbb{D}$ forms a normed linear space with respect to the hyperbolic valued norm(modulus).        
We say that a subset $ A \subset \mathbb{D}$ is a $\mathbb{D}$- bounded set if there exists M $\in\mathbb{D}^{+}$ such that $ |z|_{k}\preceq M, \forall z\in A. $\\
Let\\
         $$ A_{1} =\left\{x \in \mathbb{R}\;|\; \exists y \in \mathbb{R}, xe + ye^{\dagger} \in A  \right\},$$ and
         $$ A_{2} =\left\{y \in \mathbb{R}\;|\; \exists x \in \mathbb{R}, xe + ye^{\dagger} \in A  \right\}$$\\
         If A is a bounded set, then both $ A_{1}$ and $A_{2}$ are bounded and $$Sup_{\mathbb{D}} A = e\; Sup A_{1} + e^{\dagger}\; Sup A_{2}.$$ For details on hyperbolic numbers, we refer to \cite{DA} and \cite {MELMSbicomhol}.
 \begin{definition} \cite [page 5] {DA Kolax}\label{def1}
Let $(\Omega,\Sigma)$ be a measurable space. A function $$ P_{D}\; \colon \Sigma \longrightarrow  \mathbb{D}$$ is called a $\mathbb{D}$-valued probabilistic measure or a $\mathbb{D}$-valued probability if  
\begin{enumerate}
\item[(i)]$ P_{\mathbb{D}}(A)\succeq0;$
\item[(ii)] $P_{\mathbb{D}}(\Omega)= p $, where p takes one of the three possible values 1,\;e or  $e^{\dagger};$
\item[(iii)] given a sequence $\left\lbrace{A_{n}}\right\rbrace\subset\Sigma$ of pairwise disjoint events,\\
\begin{equation*}
P_{\mathbb{D}}\left({\bigcup}^\infty_{n=1}\right) = \sum^{\infty}_{n=1}P_{\mathbb{D}}(A_{n}).
\end{equation*}
\end{enumerate}
\end{definition}
The triplet  $(\Omega,\Sigma,P_{\mathbb{D}})$ is called a $\mathbb{D}$- probabilistic space.\\
Every $\mathbb{D}$-valued probabilistic measure can be written as\\
 $$P_{\mathbb{D}}(A)= p_{1}(A) + p_{2}(A)k = P_{1}(A)e + P_{2}(A)e^{\dagger}$$ with $$P_{1}(A)= p_{1}(A) + p_{2}(A);P_{2}(A)= p_{1}(A) - p_{2}(A).$$\\
 In fact, to define a $\mathbb{D}$-valued probabilistic measure is equivalent to consider on the same measurable space, a pair of usual $\mathbb{R}$-valued probabilistic measures.
  \begin{definition} \cite {DA Kolax}\label{def12}
 Let $(\Omega,\Sigma,P_{\mathbb{D}})$ be a $\mathbb{D}$-probabilistic space and A,\;B be two events. The conditional probability $ P_{\mathbb{D}}(A/B)$ of event A under the condition that B has happened is defined as 
 $$P_{\mathbb{D}}(A/B) = \left\{ \begin{array}{l}
\dfrac{P_{\mathbb{D}}({A \cap B})}{P_{\mathbb{D}}(B)} \ if \ P_{\mathbb{D}}(B)\succ 0 \ and \ P_{\mathbb{D}}(B) \notin \mathfrak{S}_{\mathbb{D}};\\\\
P_{\mathbb{D}}(A) \ if \ P_{\mathbb{D}}(B)=0;\\\\
\dfrac{P_{\mathbb{D}}({A \cap B})}{\lambda_{1}} e + P_{\mathbb{D}}(A) e^{\dagger} \ if \ P_{\mathbb{D}}(B) = \lambda_{1} e, \lambda_{1}>0;\\\\
P_{\mathbb{D}}(A) e + \dfrac { P_{\mathbb{D}}({A \cap B })}{\lambda_{2}} e^{\dagger} \ if \ P_{\mathbb{D}}(B) = \lambda_{2} e, \lambda_{2}>0.\\  
\end{array}\right.$$
\end{definition}
For a fixed B, with $P_{\mathbb{D}}(B)\neq 0$,\;the conditional probability verifies all the axioms of the $\mathbb{D}$- probability, that is, it defines a new $\mathbb{D}$- probabilistic measure on the measurable space $(B,\Sigma_{B})$ where $\Sigma_{B}$ is the $\sigma$-algebra of the sets of the form $A\cap B$ with $A\in \Sigma.$
\end{section}
\begin{section}{$\mathbb{D}$- valued random variables and their properties}
The topology induced by the hyperbolic valued norm on $\mathbb{D}$ generates a Borel $\sigma$-algebra $\mathfrak{B}_{\mathbb{D}}$ on $\mathbb{D}.$\\
Let $\gamma_{0}= a_{0}e + b_{0}e^{\dagger} \in\mathbb{D}^{+}$ and $z_{0}= \mu e + \nu e^{\dagger} \in\mathbb{D}.$ Then the open ball of hyperbolic radius $\gamma_{0}$ with center at $z_{0}$ is $ B(z_{0},\gamma_{0})=\left\lbrace z=z_{1}e+z_{2}e^{\dagger}\;:\; |z-z_{0}|\prec \gamma_{0}\right\rbrace.$\\
If $ \gamma_{0} $ is not a zero divisor, i.e, $ a_{0}\neq 0,\;b_{0}\neq 0 ,$ then\\ $$ B(z_{0},\gamma_{0})=\left\lbrace z=z_{1}e+z_{2}e^{\dagger}\;:\; |z_{1}-\mu|\prec a_{0},|z_{2}-\nu|\prec b_{0}\right\rbrace,$$ which is a rectangle with center at $ z_{0} $ and having sides of length $ 2\mu $ and $ 2\nu .$\\
If $ \gamma_{0} $ is a zero divisor, then we cannot define the ball in the same way because none of the inequalities $ |z_{1}-\mu|\prec 0\;\; or\;\; |z_{2}-\nu|\prec 0 $ has solutions. So, we define the ball $ B(z_{0},\gamma_{0}) $ to be the open intervals $ (\mu-a_{0}, \mu+ a_{0}) $ or $ (\nu-b_{0}, \nu+ b_{0}) $ according as $\gamma_{0}\in\mathbb{D}_{e}^{+}$ \; or \; $\gamma_{0} \in\mathbb{D}_{e^{\dagger}}^{+}$ respectively.
\begin{definition}
Let $(\Omega,\Sigma,P_{\mathbb{D}})$ be a $\mathbb{D}$- probabilistic space. A function $$ X_{\mathbb{D}}\; \colon \Omega \longrightarrow \mathbb{D} $$ such that $$ X_{\mathbb{D}}^{-1}(D)\in \Sigma $$ for every open set D  in $\mathbb{D}$ is called a $\mathbb{D}$-random variable. Thus a $\mathbb{D}$-random variable is a $\mathbb{D}$-valued function on $ \Omega $ which is measurable.   
\end{definition}
\begin{example}
Let $(\Omega,\Sigma,P_{\mathbb{D}})$ be a $\mathbb{D}$- probabilistic space and $ E\in \Sigma. $ Then the function $ \chi_{E} \; \colon \Omega \longrightarrow \mathbb{D} $ defined by 
$$\chi_{E}(w) = \left\{ \begin{array}{l}
$\thorn$ \ \;\;\; if \ w \in E\\
0 \ \;\;\; if w \notin E\\  
\end{array} \right.$$ where \thorn = 1, e or  $ e^{\dagger}.$ Then $ \chi_{E} $ is a $\mathbb{D}$-random variable on $ \Omega $ as for any open subset V of $\mathbb{D},$ we have 
$$\chi_{E}^{-1}(V) = \left\{ \begin{array}{l}
E \ \;\;\;\; if \ $\thorn$  \in V \;\; and \;\; 0 \notin V \\
E^{c} \ \;\;\; if \ $\thorn$  \notin V \;\; and \;\; 0 \in V \\ 
\Omega \ \;\;\;\;\; if \ $\thorn$  \in V \;\; and \;\; 0 \in V \\
\phi \ \;\;\;\;\; if \ $\thorn$  \notin V \;\; and \;\; 0 \notin V \\ 
\end{array} \right.$$ where $ E, E^{c}, \Omega \; and \;\; \phi $ are measurable.
\end{example}
\begin{theorem} \cite [Theorem 1.7]{WRudin}
Let  $ g \; \colon  Y \longrightarrow Z $ be a continuous function, where Y and Z are topological spaces.If X is a measurable space, if $ f \; \colon X \longrightarrow Y $ is measurable, and $ h=g \circ f ,$ then $ h \; \colon X \longrightarrow Z $ is measurable. 
\end{theorem}
Every $\mathbb{D}$-random variable $ X_{\mathbb{D}} $ can be written as $$ X_{\mathbb{D}}(w) = x_{1}(w) e + x_{2}(w) e^{\dagger} = X_{1}(w) e + X_{2}(w) e^{\dagger},$$ where $ X_{1}(w) = x_{1}(w) + x_{2}(w) $ and $ X_{2}(w) = x_{1}(w) - x_{2}(w) $ are real valued functions on $ \Omega .$\\
Taking $ X = \Omega,\; f = X_{\mathbb{D}},\; g = \Pi_{i} \; \colon \mathbb{D} \longrightarrow \mathbb{R} $ defined by $ \Pi_{i}(z = z_{1} e + z_{2} e^{\dagger}) = z_{i},\; i=1,2 $ in theorem 3.3, we find that $ g \circ f = g \circ X_{\mathbb{D}} = \Pi_{i} \circ X_{\mathbb{D}} = X_{i}  $ is measurable for i=1,\;2. Thus every $\mathbb{D}$-random variable $ X_{\mathbb{D}} $ on $ \Omega $ can be written as \\ $ X_{\mathbb{D}}(w) = X_{1}(w) e + X_{2}(w) e^{\dagger}, $ where $ X_{1}(w) $ and $ X_{2}(w) $ are real valued random variables on $ \Omega .$ 
\begin{theorem} \cite [Theorem 1.8]{WRudin}
Let u and v be real measurable functions on a measure space X, let $ \phi $ be the continuous mapping of the plane into a topological space Y, and define $ h(x) = \phi(u(x),v(x)),$ for $ x \in X.$ Then $  h \; \colon X \longrightarrow Y $ is measurable. 
\end{theorem}
Let $ X_{1} $ and $ X_{2} $ be real valued random variables on a measure space $ \Omega,  Y=\mathbb{D}, $ let $ \phi(s,t) = \alpha s + \beta t, \forall s,t \in \mathbb{R} $ and $ \alpha , \beta \in \mathbb{D}. $ Then by using theorem 3.4, we see that $ \alpha X_{1} + \beta X_{2} $ is a $\mathbb{D}$-random variable on $ \Omega.$ In particular $ e X_{1} + e^{\dagger} X_{2} $ is a $\mathbb{D}$-random variable on $ \Omega.$ 
Thus any $\mathbb{D}$-random variable $ X_{\mathbb{D}} $ determines on $ \Omega,$ two real valued random variables $ X_{1} $ and $ X_{2}$ such that $ X_{\mathbb{D}}=e X_{1} + e^{\dagger} X_{2} $ and conversely. Thus to define a $\mathbb{D}$-valued random variable is equivalent to consider on the same measurable space, a pair of usual $\mathbb{R}$-valued random variables.
\begin{theorem}
A $\mathbb{D}$- valued function $ X_{\mathbb{D}}$    on measurable space $(\Omega,\Sigma)$ is a $\mathbb{D}$-random variable if and only if   $\left\lbrace{w \in \Omega \; | \; X_{\mathbb{D}}(w) \succ \alpha \; \text {or} \; X_{\mathbb{D}}(w) \in A_{\alpha} } \right\rbrace \in \Sigma,$ for every $ \alpha \in \mathbb{D}.$ 
\end{theorem} 
\begin{proof}
Let $ X_{\mathbb{D}}$ be a $\mathbb{D}$- valued function on measurable space $(\Omega,\Sigma)$ such that $\left\lbrace{w \in \Omega \; | \; X_{\mathbb{D}}(w) \succ \alpha \; \text {or} \; X_{\mathbb{D}}(w) \in A_{\alpha} }\right\rbrace \in \Sigma,$ for every $ \alpha \in \mathbb{D}.$ Let S be the collection of all E $ \subset \mathbb{D} $ such that $ X_{\mathbb{D}}^{-1}(E) \in \Sigma.$ Choose a hyperbolic number $ \alpha ,$ and choose $ \alpha_{n} < \alpha $ so that $ \alpha_{n} \rightarrow \alpha $ as $ n \rightarrow \infty. $ Since $\left\lbrace{z \in \mathbb{D} \; | \; z \succ \alpha_{n} \; \in A_{\alpha_{n}} }\right\rbrace \in S,$ for every $ n \in \mathbb{N} $ and $$ \left\lbrace{z \in \mathbb{D} \; | \; z \prec \alpha }\right\rbrace = \cup_{n=1}^{\infty} \left\lbrace{z \in \mathbb{D} \; | \; z \preceq \alpha_{n} }\right\rbrace = \cup_{n=1}^{\infty} \left\lbrace{z \in \mathbb{D} \; | \; z \succ \alpha_{n} \; \text {or} \; z \in A_{\alpha_{n}} }\right\rbrace ^{c}, $$ and S is a $\sigma$-algebra, we see that $ \left\lbrace{z \in \mathbb{D} \; | \; z \prec \alpha }\right\rbrace \in S. $ Now any ball $ B(z_{0},\gamma_{0}) $ with centre at  $ z_{0} = z_{01} e + z_{02} e^{\dagger} $ and radius $ \gamma_{0} = \gamma_{01} e + \gamma_{02} e^{\dagger} $ can be written as $$   \left\lbrace{z \in \mathbb{D} \; | \; z \prec (z_{01}  + \gamma_{01}) e + (z_{02} + \gamma_{02}) e^{\dagger}) }\right\rbrace \cap \left\lbrace{z \in \mathbb{D} \; | \; z \succ (z_{01}  - \gamma_{01}) e + (z_{02} - \gamma_{02}) e^{\dagger}) }\right\rbrace $$ and so  $ B(z_{0},\gamma_{0}) \in S. $ Since every open set in $ \mathbb{D} $ is a countable union of balls of these types, S contains every open set. Thus $ X_{\mathbb{D}}$ is a $\mathbb{D}$-random variable.\\
Conversely, suppose that $ X_{\mathbb{D}}$ is a $\mathbb{D}$-random variable.Then $ X_{1}$ and $ X_{2}$ are real valued random variables. The set $\left\lbrace{w \in \Omega \; | \; X_{\mathbb{D}}(w) \succ \alpha \; \text {or} \; X_{\mathbb{D}}(w) \in A_{\alpha} } \right\rbrace $ is union of the sets $$\left\lbrace{w \in \Omega \; | \; X_{1}(w) > \alpha_{1} } \right\rbrace \cap \left\lbrace{w \in \Omega \; | \; X_{2}(w) > \alpha_{2} } \right\rbrace, \left\lbrace{w \in \Omega \; | \; X_{1}(w) < \alpha_{1} and X_{2}(w) > \alpha_{2} } \right\rbrace, $$  and $ \left\lbrace{w \in \Omega \; | \; X_{1}(w) > \alpha_{1} and X_{2}(w) < \alpha_{2} } \right\rbrace $ for every $ \alpha = \alpha_{1} e + \alpha_{2} e^{\dagger} $ and each of these sets is measurable as $ X_{1}$ and $ X_{2}$ are real valued random variables. Thus $\left\lbrace{w \in \Omega \; | \; X_{\mathbb{D}}(w) \succ \alpha \; \text {or} \; X_{\mathbb{D}}(w) \in A_{\alpha} } \right\rbrace \in \Sigma, $ for every $ \alpha \in \mathbb{D}.$           

\end{proof}
\begin{theorem}
Let $ X_{\mathbb{D}} $  be $\mathbb{D}$-random variable on $(\Omega,\Sigma,P_{\mathbb{D}})$ . Then the following conditions are equivalent: 
\begin{enumerate}
\item[(i)]$\left\lbrace{w \in \Omega \; | \; X_{\mathbb{D}}(w) \preceq \alpha}\right\rbrace \in \Sigma,$ for every $ \alpha \in \mathbb{D}.$\\ 
\item[(ii)] $\left\lbrace{w \in \Omega \; | \; X_{\mathbb{D}}(w) \succ \alpha \; \text {or} \; X_{\mathbb{D}}(w) \in A_{\alpha}}\right\rbrace \in \Sigma,$ for every $ \alpha \in \mathbb{D}.$\\
\item[(iii)] $\left\lbrace{w \in \Omega \; | \; X_{\mathbb{D}}(w) \succeq \alpha \; \text {or} \; X_{\mathbb{D}}(w) \in A_{\alpha}}\right\rbrace \in \Sigma,$ for every $ \alpha \in \mathbb{D}.$\\
\item[(iv)] $\left\lbrace{w \in \Omega \; | \; X_{\mathbb{D}}(w) \prec \alpha}\right\rbrace \in \Sigma,$ for every $ \alpha \in \mathbb{D}.$\\
\end{enumerate}
\end{theorem}
\begin{proof}
$(i)\Leftrightarrow (ii). $ Let $ \alpha \in \mathbb{D}.$ Let $ D_{1} $ and $ D_{2} $ be two sets in (i) and (ii) respectively. Then $ D_{1} \cup D_{2} = \mathbb{D} \in \Sigma $ and $ D_{1} \cap D_{2} = \phi $ so that $ D_{1} = \mathbb{D} \setminus D_{2} $ and $ D_{2} = \mathbb{D} \setminus D_{1}. $ Thus $ D_{1} \in \Sigma \Leftrightarrow D_{2} \in \Sigma. $ \\
 $(iii)\Leftrightarrow (iv) $ by the same argument as above.\\
$(iv)\Leftrightarrow (i). $    Note that for any $ w \in \Omega $ and $ \alpha \in \mathbb{D}, $ we have $ X_{\mathbb{D}} \preceq \alpha $ if and only if  $ X_{\mathbb{D}} \preceq \alpha + \frac{1}{n} $ for every $ n \in \mathbb{N}.$ thus we have  $$\left\lbrace{w \in \Omega \; | \; X_{\mathbb{D}}(w) \preceq \alpha}\right\rbrace = \cap_{n \in \mathbb{N}} \left\lbrace{w \in \Omega \; | \; X_{\mathbb{D}}(w) \prec \alpha + \frac{1}{n} }\right\rbrace. $$
If $ X_{\mathbb{D}} $ satisfies (iv), then every set in the intersection is $ \Sigma $ and then so is the countable intersection. Thus $ X_{\mathbb{D}} $ satisfies (i).\\
$(i)\Leftrightarrow (iv). $ Proof follows by the same argument as above.\\       
\end{proof}
If $ X_{1} $ and $ X_{2} $ are Lebesgue integrable $ \mathbb{D} $-random variables with respect to the real valued probabilistic measures $ P_{1} $ and $ P_{2} $ respectively, we define the integral of $\mathbb{D}$-random variable $ X_{\mathbb{D}} = e X_{1} + e^{\dagger} X_{2} $ as\\
$$ \int_{\Omega} X_{\mathbb{D}}\; dP_{\mathbb{D}} = e \int_{\Omega} X_{1} \; dP_{1} + e^{\dagger} \int_{\Omega} X_{2} \; dP_{2}, $$ which exists in $\mathbb{D}$ since $ \int_{\Omega} X_{1} dP_{1} $ and $ \int_{\Omega} X_{1} dP_{1} $ both exists in $\mathbb{R}.$ We say that $ X_{\mathbb{D}} $ is  $ \mathbb{D} $-Lebesgue integrable.
\begin{theorem}
Let $ X_{\mathbb{D}} $ and $ Y_{\mathbb{D}} $ be $\mathbb{D}$-random variables on $(\Omega,\Sigma,P_{\mathbb{D}})$ which are integrable. Then $ \alpha X_{\mathbb{D}} + \beta Y_{\mathbb{D}} $ is integrable and $$ \int_{\Omega}( \alpha X_{\mathbb{D}} + \beta Y_{\mathbb{D}}) \; dP_{\mathbb{D}} = \alpha \int_{\Omega} X_{\mathbb{D}} \; dP_{\mathbb{D}} + \beta \int_{\Omega} Y_{\mathbb{D}} \; dP_{\mathbb{D}}, $$ where $ \alpha $ and $ \beta $ are hyperbolic numbers.       
\end{theorem}
\begin{proof}
Let $ X_{\mathbb{D}} = e X_{1} + e^{\dagger} X_{2},  Y_{\mathbb{D}} = e Y_{1} + e^{\dagger} Y_{2}, \alpha = e \alpha_{1} + e^{\dagger} \alpha_{2},  \beta = e \beta_{1} + e^{\dagger} \beta_{2}. $ Then $ X_{1}, X_{2}, Y_{1}, Y_{2} $ are real valued random variables which are Lebesgue integrable and $ \alpha_{1}, \alpha_{2}, \beta_{1}, \beta_{2} \in \mathbb{R}. $ Therefore $ \alpha_{i} X_{i} + \beta_{i} Y_{i}, i=1,2 $ are real valued random variables which are Lebesgue integrable and so $$ e(\alpha_{1} X_{1} + \beta_{1} Y_{1}) + e^{\dagger} (\alpha_{2} X_{2} + \beta_{2} Y_{2}) = \alpha X_{\mathbb{D}} + \beta Y_{\mathbb{D}} $$ is a $\mathbb{D}$-random variable which is integrable.\\
Thus $ \alpha X_{\mathbb{D}} + \beta Y_{\mathbb{D}} $ is integrable.\\
Now, 
\begin{equation*}
\begin{split}
\int_{\Omega}( \alpha X_{\mathbb{D}} + \beta Y_{\mathbb{D}}) \; dP_{\mathbb{D}}
&= e \int_{\Omega}(\alpha_{1} X_{1} + \beta_{1} Y_{1}) \; dP_{1} + e^{\dagger} \int_{\Omega} (\alpha_{2} X_{2} + \beta_{2} Y_{2}) \; dP_{2} \\ 
&= e \left[ \alpha_{1} \int_{\Omega} X_{1} \; dP_{1} + \beta_{1} \int_{\Omega} Y_{1} \; dP_{1} \right] + e^{\dagger} \left[ \alpha_{2} \int_{\Omega} X_{2} \; dP_{2} + \beta_{2} \int_{\Omega} Y_{2} \; dP_{2} \right]  \\ 
&= \left[ e \alpha_{1} \int_{\Omega} X_{1} \; dP_{1} + e^{\dagger} \alpha_{2} \int_{\Omega} X_{2} \; dP_{2} \right] + \left[ e \beta_{1} \int_{\Omega} Y_{1} \; dP_{1} + e^{\dagger} \beta_{2} \int_{\Omega} Y_{2} \; dP_{2} \right]  \\ 
&= (e \alpha_{1} + e^{\dagger} \alpha_{2}) \left[ e \int_{\Omega} X_{1} \; dP_{1} + e^{\dagger} \int_{\Omega} X_{2} \; dP_{2} \right] \\
&+ (e \beta_{1} + e^{\dagger} \beta_{2}) \left[ e \int_{\Omega} Y_{1} \; dP_{1} + e^{\dagger} \int_{\Omega} Y_{2} \; dP_{2} \right] \\
&= \alpha \int_{\Omega} X_{\mathbb{D}} \; dP_{\mathbb{D}} + \beta \int_{\Omega} Y_{\mathbb{D}} \; dP_{\mathbb{D}}.\\
\end{split}
\end{equation*} 
\end{proof}
A $\mathbb{D}$- valued random variable which can assume only a countable number of hyperbolic values and the values which it takes depends on chance is called a discrete $\mathbb{D}$- valued random variable and a $\mathbb{D}$- valued random variable whose different values cannot be put in one-one correspondence with the set of natural numbers is called continuous $\mathbb{D}$- valued random variable. Let $ X_{\mathbb{D}} = e X_{1} + e^{\dagger} X_{2} $ be a $\mathbb{D}$- valued random variable on $(\Omega,\Sigma,P_{\mathbb{D}}).$ Then $ X_{1} $ and $X_{2}$ are real valued discrete or continuous random variables according as $ X_{\mathbb{D}} $ is discrete or continuous $\mathbb{D}$- valued random variable respectively. The functions $ f_{1} (x)= P_{1}(X_{1}= x) $ and $ f_{2} (y)= P_{2}(X_{2}= y) $ are real valued probability functions of $ X_{1} $ and $X_{2}$ respectively, where $ P_{1} $ and $ P_{2} $ are real valued probabilistic measures on $(\Omega,\Sigma) $ such that $ P_{\mathbb{D}}(A)= e P_{1} (A) + e^{\dagger} P_{2} (A), $ for all $ A \in \Sigma. $\\ The function $ f_{\mathbb{D}}\colon \mathbb{D} \;\longrightarrow \mathbb{D} $ defined as $ f_{\mathbb{D}} (z= e \;x + e^{\dagger}\; y)= [e\; f_{1}(x) + e^{\dagger}\; f_{2} (y)]\; \text {\thorn}, $ for all $ z \in \mathbb{D} $ where \thorn = 1, e or $ e^{\dagger}$ is called probability function of $ X_{\mathbb{D}}. $   If $ X_{\mathbb{D}} $ is discrete, then $ f_{\mathbb{D}}$ is called $\mathbb{D}$- valued probability mass function of $X_{\mathbb{D}} $ and if $ X_{\mathbb{D}} $ is continuous, then $ f_{\mathbb{D}}$ is called $\mathbb{D}$- valued probability density function of $X_{\mathbb{D}}. $ The functions $ F_{1} (x)= P_{1}(X_{1}\leq x) $ and $ F_{2} (y)= P_{2}(X_{2}\leq y) $ are real valued distribution functions of $ X_{1} $ and $X_{2}$ respectively and the function $ F_{\mathbb{D}}\colon \mathbb{D} \;\longrightarrow \mathbb{D} $ defined as\\ $ F_{\mathbb{D}} (z= e\; x + e^{\dagger}\; y)= [e\; F_{1}(x) + e^{\dagger}\; F_{2} (y)]\; \text {\thorn}, $ for all $ z \in \mathbb{D} $ where \thorn = 1, e or $ e^{\dagger}, $ is called distribution function of $X_{\mathbb{D}}. $                  
\end{section}
\begin{section}{ Expectation and variance of $\mathbb{D}$- valued random variables}
In this section, we define expectation, variance and moment generating function of a $\mathbb{D}$- valued random variable and study their properties.
\begin{definition}\label{def32}
The expectation of a $\mathbb{D}$- valued random variable $X_{\mathbb{D}}$ on $(\Omega,\Sigma,P_{\mathbb{D}}),$ denoted by $E(X_{\mathbb{D}}),$ is defined as $$E(X_{\mathbb{D}})=\int_{\Omega}X_{\mathbb{D}}\; dP_{\mathbb{D}},$$ provided $X_{\mathbb{D}}$ is $\mathbb{D}$-Lebesgue integrable. 
\end{definition}
The expectation of a $\mathbb{D}$- valued random variable $X_{\mathbb{D}}=eX_{1}+e^{\dagger}X_{2}$ can be written as $$E(X_{\mathbb{D}})= \int_{\Omega}X_{\mathbb{D}}\; dP_{\mathbb{D}} = e \int_{\Omega} X_{1} \; dP_{1} + e^{\dagger} \int_{\Omega} X_{2} \; dP_{2}= e \; E(X_{1})+e^{\dagger}\; E(X_{2}).$$ It is also called mean of $X_{\mathbb{D}}$ and is denoted by $\mu_{\mathbb{D}}.$ 
\begin{definition}\label{def3}
The variance of a $\mathbb{D}$- valued random variable $X_{\mathbb{D}}$ on $(\Omega,\Sigma,P_{\mathbb{D}}),$ denoted by $var(X_{\mathbb{D}}),$ is defined as $$var(X_{\mathbb{D}})= E(X_{\mathbb{D}}-\mu_{\mathbb{D}})^{2}=\int_{\Omega}(X_{\mathbb{D}}-\mu_{\mathbb{D}})^{2}\; dP_{\mathbb{D}}.$$
The variance of a $\mathbb{D}$- valued random variable $$X_{\mathbb{D}}=eX_{1}+e^{\dagger}X_{2}$$ can be written as $$var(X_{D})=e\; var(X_{1})+e^{\dagger}\; var(X_{2}).$$
\end{definition}  
\begin{example}\label{exm1}
Let $(\Omega,\Sigma,P_{\mathbb{D}})$ be a $\mathbb{D}$-probabilistic space. The function $$ \chi_{A}\; \colon \Omega \longrightarrow  \mathbb{D}$$ defined by 
$$\chi_{A}(w)=\left\{  \begin{array}{l}  
$\thorn$ \ \;\;\; if \ w \in A \\
0 \ \;\;\; if w \notin A,\\  
\end{array}\right. $$ is a $ \mathbb{D}$-random variable called indicator $ \mathbb{D}$- random variable,\\
where \thorn =1,\;e\; or\; $ e^{\dagger}.$ Then
\begin{equation*}
\begin{split}
 \mu_{\mathbb{D}}=E(\chi_{A})& = \int_{\Omega}\chi_{A}\; dP_{\mathbb{D}}\\
& =0\;P_{\mathbb{D}}(\chi_{A}=0)+ \text{\thorn} \; P_{\mathbb{D}}(\chi_{A}=p)\\
&=0\;P_{\mathbb{D}}(A^{c})+ \text{\thorn}\;P_{\mathbb{D}} (A)\\
&=0+ \text{\thorn}\;P_{\mathbb{D}}(A)\\
&=\text{\thorn}\;P_{\mathbb{D}}(A)\\
&=P_{\mathbb{D}}(A),\;e\;P_{\mathbb{D}}(A)\;or\;e^{\dagger}\;P_{\mathbb{D}}(A) 
\end{split}
\end{equation*}
according as \thorn =1,\;e\; or \;$ e^{\dagger}$ respectively.
\end{example}
This example shows that expectation of a $\mathbb{D}$-random variable can be zero divisor. If we want to talk about more than one random variable simultaneously, we need joint distributions.\\
Suppose $ X_{\mathbb{D}}$ and $ Y_{\mathbb{D}}$ are discrete $\mathbb{D}$-random variables with ranges $R_{X_{\mathbb{D}}}$ and $R_{Y_{\mathbb{D}}}$ respectively defined on some $\mathbb{D}$-probabilistic space.Then their joint distribution is a function $h \colon R_{X_{\mathbb{D}}}\times R_{Y_{\mathbb{D}}} \longrightarrow \mathbb{D}$  defined by $$ h(\alpha,\beta)= P_{\mathbb{D}}(X_{\mathbb{D}}=\alpha, Y_{\mathbb{D}}=\beta),\forall (\alpha,\beta) \in R_{X_{\mathbb{D}}}\times R_{Y_{\mathbb{D}}}.$$ The marginal probability functions of $ X_{\mathbb{D}}$ and $ Y_{\mathbb{D}}$ are given respectively as follows: 
 $$ f_{X_{\mathbb{D}}}(\alpha)=\left\{  \begin{array}{l}
\Sigma_{y}h(\alpha,\beta),\;\;\;\ if \ X_{\mathbb{D}},\; Y_{\mathbb{D}} \;\text {are discrete} \\
\int_{\mathbb{D}} h(\alpha,\beta)\;dy \ if \ X_{\mathbb{D}},\; Y_{\mathbb{D}}\; \text{are continuous} \\
\end{array}\right.$$ and
$$ f_{Y_{\mathbb{D}}}(\beta)=\left\{  \begin{array}{l}
\Sigma_{x}h(\alpha,\beta),\;\;\;\ if \ X_{\mathbb{D}},\; Y_{\mathbb{D}} \;\text {are discrete} \\
\int_{\mathbb{D}} h(\alpha,\beta)\;dx \ if \ X_{\mathbb{D}},\; Y_{\mathbb{D}}\; \text{are continuous}. \\
\end{array}\right.
 $$ Two $\mathbb{D}$-random variables $ X_{\mathbb{D}}$ and $ Y_{\mathbb{D}}$ are said to be independent if $$ h(\alpha,\beta)=f_{X_{\mathbb{D}}}(\alpha)f_{Y_{\mathbb{D}}}(\beta),\; \forall (\alpha,\beta) \in R_{X_{\mathbb{D}}}\times R_{Y_{\mathbb{D}}}. $$  The joint distribution function of two random variables $X_{\mathbb{D}}= e X_{1} + e^{\dagger} X_{2}$ and $Y_{\mathbb{D}}=e Y_{1} + e^{\dagger} Y_{2}$ , denoted by $F_{X_{\mathbb{D}}Y_{\mathbb{D}}}(\alpha, \beta)$ is a hyperbolic valued function defined for all $\alpha= e \alpha_{1} + e^{\dagger} \alpha_{2}, \beta= e \beta_{1} + e^{\dagger} \beta_{2} \in \mathbb{D}$ by $$F_{X_{\mathbb{D}}Y_{\mathbb{D}}}(\alpha, \beta)= e\; F_{X_{1}Y_{1}}(\alpha_{1}, \beta_{1})+ e^{\dagger}\; F_{X_{2}Y_{2}}(\alpha_{2}, \beta_{2}).$$ where $ F_{X_{1}Y_{1}}$ and $F_{X_{2}Y_{2}}$ are joint distribution functions of real valued random variables $ X_{1}, Y_{1}$ and $ X_{2}, Y_{2}$ respectively.  
\end{section}            

\begin{section}{Properties of expectation and variance}
 \begin{enumerate}
 \item[(i)]The addition theorem. If  $X_{\mathbb{D}}$ and $Y_{\mathbb{D}}$ are $\mathbb{D}$-random variables, then $$E(X_{\mathbb{D}}+Y_{\mathbb{D}})=E(X_{\mathbb{D}})+E(Y_{\mathbb{D}}).$$  
   \begin{proof}
   Since $X_{\mathbb{D}}$ and $Y_{\mathbb{D}}$ are $\mathbb{D}$-random variables, therefore,$$X_{\mathbb{D}}= e X_{1}+ e^{\dagger} X_{2}\; and \; Y_{\mathbb{D}}= e Y_{1}+ e^{\dagger} Y_{2}, $$where $X_{1},X_{2}, Y_{1},Y_{2}$ are real valued random variables.\\Now $X_{\mathbb{D}}+Y_{\mathbb{D}}=e (X_{1}+Y_{1})+ e^{\dagger} (X_{2}+Y_{2})$ and so  
   \begin{equation*}
   \begin{split}
    E(X_{\mathbb{D}}+Y_{\mathbb{D}})& =E[e (X_{1}+Y_{1})+ e^{\dagger} (X_{2}+Y_{2})]\\
   &=eE(X_{1}+Y_{1})+ e^{\dagger}E (X_{2}+Y_{2})\\
   &=e[E(X_{1})+E(Y_{1})] + e^{\dagger}[E(X_{2})+E(Y_{2})]\\
   &=[eE(X_{1})+e^{\dagger}E(X_{2})]+[eE(Y_{1})+e^{\dagger}E(Y_{2})]\\
   &=E(X_{\mathbb{D}})+E(Y_{\mathbb{D}}).\\
   \end{split}
   \end{equation*}
    \end{proof}
  \item[(ii)] If  $X_{\mathbb{D}}$ is a $\mathbb{D}$-random variable and a , b are hyperbolic numbers, then $$E(aX_{\mathbb{D}}+b)= a E(X_{\mathbb{D}})+b.$$
  \begin{proof}
  Let \; $ X_{\mathbb{D}}= e X_{1}+ e^{\dagger} X_{2},a=ea_{1}+e^{\dagger}a_{2},b=eb_{1}+e^{\dagger}b_{2}.$
  
  \begin{equation*}
     \begin{split}
    Then\;\;\; E(aX_{\mathbb{D}}+b)&=E[(ea_{1}+e^{\dagger}a_{2})( e X_{1}+ e^{\dagger} X_{2})+(eb_{1}+e^{\dagger}b_{2})]\\
     &=E[(ea_{1}X_{1}+e^{\dagger}a_{2}X_{2})+(eb_{1}+e^{\dagger}b_{2})]\\
     &=E[e(a_{1}X_{1}+b_{1})+e^{\dagger}(a_{2}X_{2}+b_{2})]\\
     &=eE(a_{1}X_{1}+b_{1})+e^{\dagger}E(a_{2}X_{2}+b_{2})\\
     &=e[a_{1}E(X_{1})+b_{1}]+e^{\dagger}[a_{2}E(X_{2})+b_{2}] \textmd{ as $X_{1}$ and $X_{2}$ are real}\\
     & \textmd{ valued random variables and $a_{1},a_{2},b_{1},b_{2}$ are real numbers.}
     \end{split}
     \end{equation*}
     \begin{equation*}
          \begin{split}
          Therefore,\;
          E(aX_{\mathbb{D}}+b)&=e[a_{1}E(X_{1})+b_{1}]+e^{\dagger}[a_{2}E(X_{2})+b_{2}]\\
         &=(ea_{1}+e^{\dagger}a_{2})[eE(X_{1})+e^{\dagger}E(X_{2})]+(eb_{1}+e^{\dagger}b_{2})\\
         & = a E(X_{\mathbb{D}})+b.\\
          \end{split}
          \end{equation*}
          \end{proof}
       \item[(iii)] If  $X_{\mathbb{D}}$ and $Y_{\mathbb{D}}$ are $\mathbb{D}$-random variables and a, b are hyperbolic numbers, then 
       $$E(aX_{\mathbb{D}}+bY_{\mathbb{D}})=aE(X_{\mathbb{D}})+bE(Y_{\mathbb{D}}).$$That is, E is linear operator.
       This result can be generalized as if $X_{\mathbb{D}}^{1},X_{\mathbb{D}}^{2},\dots , X_{\mathbb{D}}^{n}$ are $\mathbb{D}$-random variables and $a_{1},a_{2}, \dots, a_{n}$ are hyperbolic numbers, then$$ E\left(\sum^{n}_{i=1}a_{i}X_{i}\right)= a_{i}\sum^{n}_{i=1}E(X_{i})$$
       \item[(iv)] The Product theorem. If  $X_{\mathbb{D}}$ and $Y_{\mathbb{D}}$ are $\mathbb{D}$-random variables,then $$E(X_{\mathbb{D}}.Y_{\mathbb{D}})=E(X_{\mathbb{D}}).E(Y_{\mathbb{D}}).$$
       \begin{proof}
      Let $X_{\mathbb{D}}=eX_{1}+e^{\dagger}X_{2}$ and
$Y_{\mathbb{D}}=eY_{1}+e^{\dagger}Y_{2}.$ Then
 \begin{equation*}
 \begin{split}
 E(X_{\mathbb{D}}.Y_{\mathbb{D}})
 & =E(eX_{1}Y_{1}+e^{\dagger}X_{2}Y_{2})\\
 & =eE(X_{1}.Y_{1})+e^{\dagger}E(X_{2}.Y_{2})\\
 & =e E(X_{1})E(Y_{1})+e^{\dagger}E(X_{2})E(Y_{2}),\\      
\end{split}
\end{equation*}
as $X_{i},Y_{i}$ are independent for i=1,2.\\
Therefore,\\$E(X_{\mathbb{D}}.Y_{\mathbb{D}})= [eE(X_{1})+e^{\dagger}E(X_{2})][E(X_{1})+e^{\dagger}E(X_{2})]= E(X_{\mathbb{D}}.Y_{\mathbb{D}}).$
This result can also be generalized for any number of variables.

\end{proof}
 \item[(v)] If  $X_{\mathbb{D}}$ is a $\mathbb{D}$-random variable, then $var(aX_{\mathbb{D}}+b)= a^{2}.var(X_{\mathbb{D}})$, where a and b are constants.
 \begin{proof}
 $var(aX_{\mathbb{D}}+b)= e\; var(a_{1}X_{1}+b_{1})+ e^{\dagger}\; var(a_{2}X_{2}+b_{2}),$ where$$ a= ea_{1}+ e^{\dagger} a_{2},\; b=eb_{1}+ e^{\dagger} b_{2} \text{ and}\;\;  X_{\mathbb{D}}=e X_{1} + e^{\dagger}X_{2}.$$
 Therefore,
 \begin{equation*}
 \begin{split}
var(aX_{\mathbb{D}}+b)
&=e a_{1}^{2}\; var(X_{1})+ e^{\dagger} a_{2}^{2}\; var(X_{2})\\
&=(e a_{1}^{2}+ e^{\dagger} a_{2}^{2})(e var(X_{1})+ e^{\dagger} var(X_{2}))\\
&= a^{2}\; var(X_{\mathbb{D}}).\\
 \end{split}
 \end{equation*}
  \end{proof}
\item[(vi)] If  $X_{\mathbb{D}}$ and $Y_{\mathbb{D}}$ are independent ${\mathbb{D}}$-random variables, then $$var(X_{\mathbb{D}}\pm Y_{\mathbb{D}})= var(X_{\mathbb{D}})\pm var(Y_{\mathbb{D}}).$$
\begin{proof}
\begin{equation*}
\begin{split}
var(X_{\mathbb{D}}\pm Y_{\mathbb{D}})
&=var[e( X_{1} \pm Y_{1})+ e^{\dagger}( X_{1} \pm Y_{1})]\\
&= e \; var( X_{1} \pm Y_{1})+ e^{\dagger}\;var( X_{2} \pm Y_{2})\\
&= e\;[ var(X_{1})\pm var(Y_{1})] + e^{\dagger}\; [var(X_{2})\pm var(Y_{2})]\\
&= [e\; var(X_{1})\pm e^{\dagger}\;var(X_{2})] \pm [e\; var(Y_{1})+ e^{\dagger}\;var(Y_{2})]\\
&=var(X_{\mathbb{D}})\pm var(Y_{\mathbb{D}}).\\
\end{split}
\end{equation*}
\end{proof}
\item[(vii)] Cauchy Schwartz inequality. If  $X_{\mathbb{D}}$ and $Y_{\mathbb{D}}$ are ${\mathbb{D}}$-random variables, then $$[E(X_{\mathbb{D}}Y_{\mathbb{D}})]^{2}\preceq E(X_{\mathbb{D}}^{2})E(Y_{\mathbb{D}}^{2}).$$
\begin{proof}
Let $X_{\mathbb{D}}=eX_{1}+e^{\dagger}X_{2}$ and
$Y_{\mathbb{D}}=eY_{1}+e^{\dagger}Y_{2}.$ Consider a hyperbolic valued function of hyperbolic variable h, defined by
\begin{equation*}
\begin{split}
 Z(h)=E(X_{\mathbb{D}}+h\;Y_{\mathbb{D}})^{2}
 &= e E(X_{1}+h_{1}\; Y_{1})^{2} +e^{\dagger} E(X_{2}+h_{2}\; Y_{2})^{2}\\
& = e\; Z(h_{1})+e^{\dagger}\;Z(h_{2}),\forall\; h= e h_{1}+e^{\dagger}h_{2}.\\
\end{split}
\end{equation*}
Clearly $Z(h)\in \mathbb{D}^{+},\forall\; h$ as $ (X_{\mathbb{D}}+h Y_{\mathbb{D}})^{2} \in \mathbb{D}^{+},\forall X_{\mathbb{D}},Y_{\mathbb{D}}$ and h.\\
Therefore,
$$ Z(h_{i})\geq 0, for \: i=1,2.$$
Now $Z(h_{i}),i=1,2$ are given by
\begin{equation*}
\begin{split}
Z(h_{i})
&=E(X_{i}^{2}+2h_{i}X_{i}Y_{i}+h_{i}^{2}Y_{i}^{2})\\
&= E(X_{i}^{2}) + 2h_{i}E(X_{i} Y_{i})+ h_{i}^{2}E(Y_{i}^{2}),\\
\end{split}
\end{equation*}
which are quadratic polynomials in $h_{1}$ and $h_{2}$ respectively.
This implies that the graphs of real polynomial functions $Z(h_{1})$ and $Z(h_{2})$ lie on or above $h_{1}$ and $h_{2}$-axis respectively.\\
Therefore,
 $$[2 E(X_{i} Y_{i})]^{2}-4 E(X_{i}^{2})E(Y_{i}^{2})\leq 0, i=1,2.$$ This implies that $$[ E(X_{i} Y_{i})]^{2}\leq E(X_{i}^{2})E(Y_{i}^{2}), i=1,2.$$
 and so\;\; $[E(X_{\mathbb{D}}Y_{\mathbb{D}})]^{2}\preceq E(X_{\mathbb{D}}^{2})E(Y_{\mathbb{D}}^{2}).$
 \end{proof}   
\begin{remark}
\item[(i)] Equality holds in above inequality
 \begin{equation*}
 \begin{split}
 \textmd{iff}\;\;E(X_{\mathbb{D}} + h Y_{\mathbb{D}})^{2}=0,\forall h.\\
 \textmd{iff}\;\;P_{\mathbb{D}}[(X_{\mathbb{D}} + h Y_{\mathbb{D}})^{2}=0] = \text{\thorn},\textmd{where}\;\; \text{\thorn}=1,e\;\;\textmd{or}\;\; e^{\dagger}.\\
 \textmd{iff}\;\;e P_{1}(X_{1} + h_{1} Y_{1}=0)+ e^{\dagger} P_{2}(X_{2} + h_{2} Y_{2}=0) = p.\\
 \end{split}  
 \end{equation*}
 \end{remark} 

For p=1, we have $P_{i}(X_{i} + h_{i} Y_{i}=0)=1$ \textmd{implies that} $P_{i}(Y_{i}=\frac{-X_{i}}{Y_{i}})=1$
Therefore,$$P_{i}(Y_{i} = k_{i} X_{i})=1,\;,i=1,2\; ,\textmd{where}\; k_{i}=\frac{-1}{h_{i}}.$$
For p=e, we have $P_{1}(X_{1} + h_{1} Y_{1}=0)=1$ and $P_{2}(X_{2} + h_{2} Y_{2}=0)=0$.This implies that $P_{1}(Y_{1} = k_{1} X_{1})=1$ and $P_{2}(Y_{2} = k_{2} X_{2})=0.$\\
For $p=e^{\dagger},$ we have $P_{1}(Y_{1} = k_{1} X_{1})=0$ and $P_{2}(Y_{2} = k_{2} X_{2})=1.$
\item[(ii)] Replacing $X_{\mathbb{D}}$ by $X_{\mathbb{D}} - E(X_{\mathbb{D}})$ and taking $Y_{\mathbb{D}}=1,$ we have 
$$[E(X_{\mathbb{D}} - E(X_{\mathbb{D}}))]^{2}\preceq [E(X_{\mathbb{D}} - E(X_{\mathbb{D}}))]^{2}.E(1)=p. \, var(X_{\mathbb{D}})$$ as $ E(1)=\int_{\Omega}1.dP_{\mathbb{D}}=P_{\mathbb{D}}(\Omega)=p.$ \\
Therefore, $ (\text{mean deviation about mean})^{2} \preceq p \, var(X_{\mathbb{D}}),$ which yields that $M.D.\preceq p. \, S.D.$
\end{enumerate}
\end{section}


\begin{section}{Moments and Moment Generating function}
Now we define moments and moment generating function of a $ {\mathbb{D}}$- random variable.
\begin{definition}\label{def5}
The rth moment of a \; $ {\mathbb{D}}$- random variable about any point $X_{\mathbb{D}}=a,$ denoted by $\mu_{r}^{'}$ is given by
$\mu_{r}^{'}=E(X_{\mathbb{D}}-a)^{r}$ and rth moment of $X_{\mathbb{D}}$  about mean $ \mu_{\mathbb{D}} =E(X_{\mathbb{D}}),$ denoted by $\mu_{r}$ is given by
$\mu_{r}=E(X_{\mathbb{D}}-\mu_{\mathbb{D}})^{r}.$
\end{definition}
  
  \begin{definition}\label{def6}
  Let $X_{\mathbb{D}}=eX_{1}+e^{\dagger}X_{2}$ be a $ {\mathbb{D}}$- random variable taking countable values, where $X_{1}$ and $X_{2}$ are real random variables taking countable values with moment generating functions $ M_{X_{1}}(t_{1})$ and $ M_{X_{2}}(t_{2})$ respectively.Then moment generating function of $X_{\mathbb{D}}$ (if it exists), denoted by $ M_{X_{\mathbb{D}}}(t),$ is defined as $$ M_{X_{\mathbb{D}}}(t)= e M_{X_{1}}(t) + e^{\dagger}M_{X_{2}}(t),$$ where $ t=et_{1} + e^{\dagger}t_{2}.$ 
  \end{definition}
  We have $$ M_{X_{\mathbb{D}}}(t)= e \sum^{\infty}_{k=0}\frac {t_{1}^{k}}{k!}\lambda_{k}^{'} +  e^{\dagger} \sum^{\infty}_{k=0}\frac {t_{2}^{k}}{k!}\delta_{k}^{'}, $$ where $\lambda_{k}^{'}$ and $\delta_{k}^{'}$ are kth moments about origin of $X_{1}$  and $X_{2}$ respectively.\\
  
  Therefore,
  \begin{equation} {\label{m.g.f}} 
    M_{X_{\mathbb{D}}}(t)
    = \sum^{\infty}_{k=0}\frac { e t_{1}^{k} + e^{\dagger} t_{2}^{k}}{k!} (e \lambda_{k}^{'} +  e^{\dagger} \delta_{k}^{'})
    = \sum^{\infty}_{k=0}\frac {t^{k}}{k!} \mu_{k}^{'},
  \end{equation} where $t^{k}=e t_{1}^{k}+ e^{\dagger} t_{2}^{k},\; \mu_{k}^{'}= e \lambda_{k}^{'}+ e^{\dagger}\delta_{k}^{'}$ is the kth moment about origin of $ X_{\mathbb{D}},$ k=1,2,3, $\dots$
  Thus we see that the coefficient of $\frac {t_{1}^{k}}{k!}$ in $ M_{X_{\mathbb{D}}}(t)$ gives $\mu_{k}^{'}.$ Since 
  $ M_{X_{\mathbb{D}}}(t)$ generates moments, it is known as moment generating function.\\
  Differentiating (\ref{m.g.f}) with respect to t and putting t=0, we get$$ [\frac{d^{r}}{dt^{r}}{M_{X_{\mathbb{D}}}(t)}]_{t=0}=[\frac {{\mu_{r}}^{'}}{r!}.r! + t {\mu_{r+1}}^{'}+ \frac{t^{2}}{2!}{\mu_{r+1}}^{'}+ \dots]_{t=0}={\mu_{r}}^{'}$$
  This implies that
 $$ {\mu_{r}}^{'}=[\frac{d^{r}}{dt^{r}}{M_{X_{\mathbb{D}}}(t)}]_{t=0}.$$The moment generating function $ M_{X_{\mathbb{D}}}(t)$ of a $ \mathbb{D}$-random variable satisfies the following properties:
 \begin{enumerate}
 \item[(i)]$ M_{h{X_{\mathbb{D}}}}(t)=  M_{X_{\mathbb{D}}}(ht),$ where $h=eh_{1}+e^{\dagger}h_{2}$ is a hyperbolic number.
 \begin{proof}
 $ M_{h{X_{\mathbb{D}}}}(t)= e M_{h_{1}{X_{1}}}(t_{1}) + e^{\dagger} M_{h_{2}{X_{2}}}(t_{2}) =  e M_{X_{1}}(h_{1}t_{1}) + e^{\dagger} M_{h_{2}{X_{2}}}(h_{2}t_{2}).$ 
 Therefore,$$ M_{h{X_{\mathbb{D}}}}(t)=  M_{X_{\mathbb{D}}}(ht).$$ 
 \end{proof}
 \item[(ii)]The moment generating function of sum of number of independent $ \mathbb{D}$-random variables is equal to the product of their respective moment generating functions.
 \begin{proof}
  Let $ X_{\mathbb{D}}^{1}, X_{\mathbb{D}}^{2},\dots  X_{\mathbb{D}}^{n}$ be $ \mathbb{D}$-random variables which are independent.Then $ X_{\mathbb{D}}^{k}=e X_{1}^{k}+ e^{\dagger} X_{2}^{k},$ k=1,2,$\dots ,n.$
  Therefore,
   $$ \sum^{n}_{k=1}X_{\mathbb{D}}^{k}
    =\sum^{n}_{k=1}(e X_{1}^{k}+ e^{\dagger} X_{2}^{k})
    =e\sum^{n}_{k=1}X_{1}^{k}+e^{\dagger}\sum^{n}_{k=1}X_{2}^{k}.$$
    Now
    \begin{equation*}
    \begin{split}
              M_{\sum^{n}_{k=1}X_{\mathbb{D}}^{k}}(t)
      & = e M_{ \sum^{n}_{k=1} X_{1}^{k}}(t_{1})+
       e M_{ \sum^{n}_{k=1} X_{2}^{k}}(t_{2})\\
      & = e\prod^{n}_{k=1}M_{X_{1}^{k}}(t_{1})+ e^{\dagger}\prod^{n}_{k=1}M_{X_{2}^{k}}(t_{2})\\
   \end{split}
    \end{equation*}
    because $ M_{\sum^{n}_{i=1}X_{i}}(t)= \prod^{n}_{i=1}M_{X_{i}}(t)$ if $X_{1},X_{2},\dots,X_{n}$ are real independent random variables.Thus
    $$ M_{\sum^{n}_{k=1}X_{\mathbb{D}}^{k}}(t)=\prod^{n}_{k=1} [e M_{X_{1}}^{k}(t_{1})+ e^{\dagger} M_{X_{2}}^{k}(t_{2})]= \prod^{n}_{k=1}M_{X_{\mathbb{D}}}^{k}(t).$$
    \item[(iii)] The moment generating function of a $ \mathbb{D}$-random variable, if it exists, uniquely determines the distribution of $ \mathbb{D}$- random variable. Hence $ M_{X_{\mathbb{D}}}(t)=M_{Y_{\mathbb{D}}}(t)$ iff  $X_{ \mathbb{D}}$ and $Y_{ \mathbb{D}}$ are identically distributed.
    
   \end{proof}  
 \end{enumerate} 
\end{section}
 \begin{section}{Hyperbolic valued probability distributions}
 
 \subsection*{Hyperbolic Bernoulli distribution}
 Let $X_{1}$ and $X_{2}$ be two Bernoulli variables on the same measurable space $(\Omega,\Sigma).$ Let $f_{1}(x_{1})=(1-p_{1})^{1-x_{1}}p_{1}^{x_{1}}$ and $f_{2}(x_{2})=(1-p_{2})^{1-x_{2}}p_{2}^{x_{2}}$ be their respective probability functions, $ 0\leq p_{i} \leq 1,i=1,2$ and $x_{i}=0,1$ for i=1,2.\\
 Then the $ \mathbb{D}$- random variable $X_{\mathbb{D}}=(eX_{1}+e^{\dagger}X_{2}) \text{\thorn}$ having $ \mathbb{D}$-probability function given by
 $$f_{\mathbb{D}}(X_{\mathbb{D}}=x)=[e f_{1}(x)+e^{\dagger}f_{2}(x)] \text{\thorn}$$ is called $ \mathbb{D}$-Bernoulli variable and the probability distribution of $X_{\mathbb{D}}$ is called hyperbolic bernoulli distribution,where \thorn \;  takes one of the values 1,e or $e^{\dagger}.$     
  \begin{remark}
  \begin{enumerate}
\item[(i)]  If $ X_{1}=X_{2}=X_{\mathbb{D}}$ and \thorn=1,\;then$ X_{\mathbb{D}}$ is a real valued bernoulli variable.
\item[(ii)] If \thorn = e, then $ X_{\mathbb{D}}=e X_{1}$ and $P_{\mathbb{D}}(X_{\mathbb{D}}=x)=eP_{1}(x).$
\item[(iii)]If \thorn = $e^{\dagger},$ then  $X_{\mathbb{D}}=e^{\dagger} X_{2}$ and $P_{\mathbb{D}}(X_{\mathbb{D}}=x)=e^{\dagger}P_{2}(x).$ 
  \end{enumerate}
  \end{remark}
  \subsection*{Moments of hyperbolic bernoulli distribution}
  The rth moment about origin is 
  \begin{equation*}
  \mu_{r}^{'}=E({X_{\mathbb{D}}}^{r})=E[(e{X_{1}}^{r}+e^{\dagger}{X_{2}}^{r}) \text{\thorn}^{r}]\\
  = e \text{\thorn}^{r} E({X_{1}}^{r})+ e^{\dagger} \text{\thorn}^{r} E({X_{2}}^{r})\\
  = e \text{\thorn}^{r} p_{1}+ e^{\dagger} \text{\thorn}^{r} p_{2}\\
  \end{equation*}
  Therfore,\\
 $$ \mu_{r}^{'} = \text{\thorn}^{r} (e p_{1}+e^{\dagger} p_{2}), \textmd{where}\; E(X_{i}) = p_{i},i=1,2.$$
 \begin{enumerate}
 \item[(i)] If \thorn = 1,\; then $ \mu_{r}^{'} = e p_{1}+e^{\dagger} p_{2}.$
 Therefore,
$$ \mu_{1}^{'} = e p_{1}+e^{\dagger} p_{2}\; \textmd{and} \; \mu_{2}^{'} = e p_{1}+e^{\dagger} p_{2}. $$ This gives $$ E(X_{\mathbb{D}})= \mu_{1}^{'} = e p_{1}+e^{\dagger} p_{2} =e E(X_{1}) + e^{\dagger} E(X_{2})$$ and
 \begin{equation*}
 \begin{split}
  var(X_{\mathbb{D}})
  & = \mu_{2}^{'}- { \mu_{1}^{'}}^{2}\\
  & = e p_{1}(1-p_{1})+e^{\dagger} p_{2}(1-p_{2})\\
  & = e p_{1} q_{1}+e^{\dagger}       p_{2} q_{2}\\
& = e var(X_{1})+e^{\dagger} var(X_{2}),\\
\end{split}
\end{equation*}
where $ q_{i}=1-p_{i},$ i=1,2. 
 \item[(ii)] If \thorn = e, then $X_{\mathbb{D}} = e X_{1}$ and $ \mu_{r}^{'}=e p_{1}.$This gives $$E(X_{\mathbb{D}})= e. p_{1}= e E(_{1})\; \textmd{and}\; var(X_{\mathbb{D}})= e( p_{1}-{p_{1}}^{2})=e p_{1}(1-p_{1})=e p_{1} q_{1}=e. var(X_{1}).$$
 \item[(iii)] If \thorn $= e^{\dagger},$ then $X_{\mathbb{D}} = e^{\dagger} X_{2}.$ This gives $$E(X_{\mathbb{D}})= e^{\dagger}. p_{2}= e^{\dagger}E(X_{2})\; \textmd{and}\; var(X_{\mathbb{D}})= e^{\dagger}. var(X_{2}).$$
 \end{enumerate}
 \subsection*{Hyperbolic Binomial distribution}
 Let $X_{1}$ and $X_{2}$ be two Binomial variables with probability functions\\  $f_{1}(x_{1})=C(n_{1},x_{1})p_{1}^{x_{1}}q_{1}^{n_{1}-x_{1}}$ and $ f_{2}(x_{2})= C(n_{2},x_{2})p_{2}^{x_{2}}q_{2}^{n_{2}-x_{2}}$ respectively, where $ x_{i}=0,1,2, \dots,n_{i},i=1,2.$ Here $ 0 \leq p_{i}\leq 1 $ and $ q_{i}=1-p_{i},i=1,2.$ The numbers $n_{i},p_{i}$ are parameters of distributions of $X_{i},$ i=1,2.\\
  The $ \mathbb{D}$- random variable $X_{\mathbb{D}}=(eX_{1}+e^{\dagger}X_{2}) \text{\thorn}$ taking $ n_{1}.n_{2} $  values in $\mathbb{D}^{+} \cup \left\lbrace 0 \right\rbrace $ is said to be $\mathbb{D}$-Binomial variable if its $\mathbb{D}$-probability function is given by
  $$f_{\mathbb{D}}(X_{\mathbb{D}}=x) = \left\{ \begin{array}{l}
  [e f_{1}(x) + e^{\dagger} f_{2}(x)] \text{\thorn} , x=e x_{1}+e^{\dagger} x_{2};\\
  0,\;\;\;\;\;\;\;\;\;\;\;\;\; otherwise.\\  
  \end{array}\right.$$
  Here $n= e n_{1} +e^{\dagger} n_{2}, p= e p_{1} + e^{\dagger} p_{2},$ are called parameters of the distribution and \thorn \; can take any of the values 1,e or $ e^{\dagger}.$
 \subsection*{Hyperbolic Poisson distribution}
 Let $X_{1}$ and $X_{2}$ be two Poisson variables with probability functions $$\displaystyle{f_{1}(x_{1})= \frac{\exp{-\lambda_{1}}{\lambda_{1}}^{x_{1}}} {{x_{1}}!}}\;\;\text{and}\;\;  \displaystyle{f_{2}(x_{2})= \frac{\exp{-\lambda_{2}}{\lambda_{2}}^{x_{2}}} {{x_{2}}!}}$$ respectively, where $x_{i}=0,1,2, \dots;\lambda_{i}>0,i=1,2.$ Then
 the $ \mathbb{D}$- random variable $X_{\mathbb{D}}=(eX_{1}+e^{\dagger}X_{2}) \text{\thorn}$ having $ \mathbb{D}$-probabilistic measure given by
  $$f_{\mathbb{D}}(X_{\mathbb{D}}=x)=[e f_{1}(x)+e^{\dagger}f_{2}(x)] \text{\thorn},\; x=e x_{1}+e^{\dagger} x_{2} \in \mathbb{D}^{+} \cup \left\lbrace 0 \right\rbrace   $$ is called $\mathbb{D}$-Poisson variable and the distribution is called Hyperbolic Poisson distribution, where \thorn =1,\;e or \;$ e^{\dagger}.$
  \begin{remark}
 For $X_{1}=X_{2}=X_{\mathbb{D}}$ and \thorn =1,these distributions are usual Binomial and Poisson distributions respectively.\\
  \end{remark}
    Hyperbolic binomial distribution tends to Hyperbolic poisson distribution under the following conditions:
  \begin{enumerate}
  \item[(i)] $n_{i}\rightarrow \infty,i=1,2.$
  \item[(ii)] $p_{i}\rightarrow 0,i=1,2.$
  \item[(iii)] $n_{i}p_{i}=\lambda_{i}$(say) is finite,\;i=1,2.
  \end{enumerate}
   Let $X_{\mathbb{D}}=eX_{1}+e^{\dagger}X_{2}$ and $Y_{\mathbb{D}}=eY_{1}+e^{\dagger}Y_{2}$ be $\mathbb{D}$-Poisson variates.Then
   \begin{equation*}
   \begin{split}
    M_{X_{\mathbb{D}} +Y_{\mathbb{D}}} &=M_{e(X_{1}+Y_{1})+e^{\dagger}(X_{2}+Y_{2})}(et_{1}+e^{\dagger}t_{2})\\
    &= e M_{X_{1}+Y_{1}}(t) + e^{\dagger} M_{X_{2}+Y_{2}}(t),\\ 
    \end{split}
   \end{equation*}
   which is moment generating function of $\mathbb{D}$-Poisson variable as $X_{1}+Y_{1}$ and
   $X_{2}+Y_{2}$ are Poisson variables.\\
   This shows that $X_{\mathbb{D}} +Y_{\mathbb{D}}$ is a $\mathbb{D}$-Poisson variable which is additive property of $\mathbb{D}$-Poisson variables. 
   \end{section}
   \begin{section}{Conditional Expectation}
   In this section, we define hyperbolic valued conditional expectation. The results and their proofs in this chapter are essentially based on the work \cite [chapter 1] {MMrao}.  
   \begin{definition}
   Let $X_{\mathbb{D}}$ be a $\mathbb{D}$-random variable on a $\mathbb{D}$-probabilistic space $(\Omega,\Sigma,P_{\mathbb{D}}).$ If an event B, with $P_{\mathbb{D}}(B)\succ 0 $ has taken place, then the conditional expectation of $ X_{\mathbb{D}}$ given B, denoted by $ E_{B}(X_{\mathbb{D}})$ or $ E(X_{\mathbb{D}}/B)$ is defined as:
   \begin{enumerate}
   \item[(i)] $E_{B}(X_{\mathbb{D}})=\frac{1}{P_{\mathbb{D}}(B)} \int_{B} X_{\mathbb{D}}(w) dP_{\mathbb{D}}(w) $ if $ P_{\mathbb{D}}(B)\succ 0 $   and $P_{\mathbb{D}}(B) \notin \mathfrak{S}_{\mathbb{D}};$
   \item[(ii)] $E_{B}(X_{\mathbb{D}})=E(X_{\mathbb{D}})$ if $  P_{\mathbb{D}}(B)=0; $
   \item[(iii)] $E_{B}(X_{\mathbb{D}})= \int_{B} \frac{X_{\mathbb{D}}(w) dP_{\mathbb{D}}(w)}{\lambda_{1}}e + E(X_{\mathbb{D}})e^{\dagger} $ if $ P_{\mathbb{D}}(B)=\lambda_{1}e, \lambda_{1}>0;$
   \item[(iv)] $E_{B}(X_{\mathbb{D}})= E(X_{\mathbb{D}})e + \int_{B} \frac{X_{\mathbb{D}}(w) dP_{\mathbb{D}}(w)}{\lambda_{2}}e^{\dagger}$ if $ P_{\mathbb{D}}(B)=\lambda_{2}e^{\dagger}, \lambda_{2}>0.$
    \end{enumerate}
    \end{definition}
   The items (iii) and (iv) are in complete agreement with (i).\\Writing $ E(X_{\mathbb{D}})= e E(X_{1}) + e^{\dagger} E(X_{2})$ and item (i) in idempotent representation reads:
   $$ E_{B}(X_{\mathbb{D}})= \frac{e}{P_{1}(B)} \int_{B} X_{1}(w) dP_{1}(w) + \frac{e^{\dagger}}{P_{2}(B)} \int_{B} X_{2}(w) dP_{2}(w)= e E_{B}(X_{1})+ e^{\dagger} E_{B}(X_{2}).$$
   Now items (iii) and (iv) read:
   \begin{equation*}
   \begin{split}
   E_{B}(X_{\mathbb{D}})
   & = \int_{B} \frac{X_{\mathbb{D}}(w) dP_{\mathbb{D}}(w)}{\lambda_{1}}e + E(X_{\mathbb{D}})e^{\dagger}\\
   & = \frac { e \int_{B} X_{1}(w) dP_{1}(w) + e^{\dagger} \int_{B} { X_{2}(w) dP_{2}(w)}}{\lambda_{1}} e + [e E(X_{1}) + e^{\dagger} E(X_{2})] e^{\dagger}\\
   & = \frac{e}{P_{1}(B)} \int_{B}  X_{1}(w) dP_{1}(w) + e^{\dagger} E(X_{2})\\
   & = e E_{B}(X_{1}) + e^{\dagger} E_{B}(X_{2})\\  
   \end{split}
   \end{equation*}
   and
   \begin{equation*}
   \begin{split}
 E_{B}(X_{\mathbb{D}})
    & = E(X_{\mathbb{D}})e + \int_{B} \frac{X_{\mathbb{D}}(w) dP_{\mathbb{D}}(w)}{\lambda_{2}}e^{\dagger}\\
    & = E(X_{1})e + \int_{B} \frac{X_{2}(w) dP_{2}(w)}{P_{2}(B)}e^{\dagger}\\
    & = e E_{B}(X_{1}) + e^{\dagger} E_{B}(X_{2}).\\
   \end{split}
   \end{equation*}
   The conditional expectation of a $ \mathbb{D}$-random variable $ X_{\mathbb{D}} = e X_{1} + e^{\dagger} X_{2}$ can be written as
   $$ E_{B}(X_{\mathbb{D}})= e E_{B}(X_{1}) + e^{\dagger} E_{B}(X_{2})$$
   \begin{theorem}
   Let A and B be two mutually exclusive events, then
   \begin{equation}
   \label{eq1}
   P_{\mathbb{D}}(A \cup B) E_{A\cup B}(X_{\mathbb{D}})=  [P_{\mathbb{D}}(A)  E_{A}(X_{\mathbb{D}}) + P_{\mathbb{D}}(B)  E_{B}(X_{\mathbb{D}})].
   \end{equation}
   \end{theorem}
   \begin{proof}
   Consider the different cases that arise.
   \begin{enumerate}
   \item[(a)] If $P_{\mathbb{D}}(A \cup B)\succ 0$ and $P_{\mathbb{D}}(A \cup B) \notin \mathfrak{S}_{\mathbb{D}},$ then
   \begin{equation}\label{eqn2}
   \begin{split}
   E_{A\cup B}(X_{\mathbb{D}})
    & = \frac{1}{P_{\mathbb{D}}(A \cup B)} \int_{A \cup B} X_{\mathbb{D}}(w) dP_{\mathbb{D}}(w)\\
    & = \frac{1}{P_{\mathbb{D}}(A \cup B)}\left[ \int_{A} X_{\mathbb{D}}(w) dP_{\mathbb{D}}(w) +  \int_{B} X_{\mathbb{D}}(w) dP_{\mathbb{D}}(w)\right].\\
   \end{split}
   \end{equation}
   \subsubitem(i) If $P_{\mathbb{D}}(A)\succ 0, P_{\mathbb{D}}(B) \succ 0$ and $P_{\mathbb{D}}(A),P_{\mathbb{D}}(B) \notin \mathfrak{S}_{\mathbb{D}},$ then from $\ref{eqn2},$ we have 
    $$ E_{A\cup B}(X_{\mathbb{D}}) = \frac{1}{P_{\mathbb{D}}(A \cup B)} [ P_{\mathbb{D}}(A)E_{A}(X_{\mathbb{D}} + P_{\mathbb{D}}(B)E_{B}(X_{\mathbb{D}})].$$
   \subitem(ii) If $P_{\mathbb{D}}(A)\succ 0, P_{\mathbb{D}}(B) = 0$ and $P_{\mathbb{D}}(A) \notin \mathfrak{S}_{\mathbb{D}},$ then \begin{equation*}
   \begin{split}
      E_{A\cup B}(X_{\mathbb{D}})
      &= \frac{1}{P_{\mathbb{D}}(A \cup B)}\left[ \int_{A} X_{\mathbb{D}}(w) dP_{\mathbb{D}}(w)+0\right]\\
      &= \frac{1}{P_{\mathbb{D}}(A \cup B)} [P_{\mathbb{D}}(A)E_{A}(X_{\mathbb{D}}) + P_{\mathbb{D}}(B)E_{B}(X_{\mathbb{D}})].\\
       \end{split}
       \end{equation*}
   \subsubitem(iii) If $P_{\mathbb{D}}(B)\succ 0, P_{\mathbb{D}}(A) = 0$ and $P_{\mathbb{D}}(B) \notin \mathfrak{S}_{\mathbb{D}},$ then
   \begin{equation*}
   \begin{split} 
      E_{A\cup B}(X_{\mathbb{D}})
    & = \frac{1}{P_{\mathbb{D}}(A \cup B)}\left[0 +  \int_{B} X_{\mathbb{D}}(w) dP_{\mathbb{D}}(w)\right]\\
    & = \frac{1}{P_{\mathbb{D}}(A \cup B)} [ P_{\mathbb{D}}(A). E_{A}(X_{\mathbb{D}}) + P_{\mathbb{D}}(B)  E_{B}(X_{\mathbb{D}}) ].\\
     \end{split}
     \end{equation*} 
     In all the cases,\\
      $$ E_{A\cup B}(X_{\mathbb{D}})=\frac{1}{P_{\mathbb{D}}(A \cup B)} [P_{\mathbb{D}}(A)E_{A}(X_{\mathbb{D}}) + P_{\mathbb{D}}(B)E_{B}(X_{\mathbb{D}})].$$
      Thus
           $$P_{\mathbb{D}}(A \cup B) E_{A\cup B}(X_{\mathbb{D}})= P_{\mathbb{D}}(A)E_{A}(X_{\mathbb{D}}) + P_{\mathbb{D}}(B)E_{B}(X_{\mathbb{D}}).$$
      \item[(b)] If $P_{\mathbb{D}}(A \cup B)= 0,$ then $ P_{\mathbb{D}}(A)= P_{\mathbb{D}}(B)= 0.$ The result holds in this case.
      \item[(c)] If $P_{\mathbb{D}}(A \cup B)= \lambda_{1}\; e$ with $\lambda_{1}> 0,$ then $P_{\mathbb{D}}(A)= \mu_{1}\; e$ and $P_{\mathbb{D}}(B)= \mu_{2}\; e$ with $\mu_{1}>0$ or $\mu_{2}>0$ or both $ \mu_{1}, \mu_{2}> 0,$ where $ \mu_{1} + \mu_{2}= \lambda_{1}.$
      \subsubitem(i) If $\mu_{1}>0$ and $\mu_{2}=0,$ then $ P_{\mathbb{D}}(B)=0.$\\
      Now
\begin{equation*}
\begin{split}
P_{\mathbb{D}}(A \cup B) E_{A\cup B}(X_{\mathbb{D}})
& = \lambda_{1} e \left[ \int_{A \cup B} \frac{X_{\mathbb{D}}(w) dP_{\mathbb{D}}(w)}{\lambda_{1}}e + E(X_{\mathbb{D}})e^{\dagger}\right]\\
& =  e  \int_{A \cup B} X_{\mathbb{D}}(w) dP_{\mathbb{D}}(w) \\ 
& = e \left[ \int_{A} X_{\mathbb{D}}(w) dP_{\mathbb{D}}(w)+  \int_{B} X_{\mathbb{D}}(w) dP_{\mathbb{D}}(w)  \right]\\
& = e  \int_{A} X_{\mathbb{D}}(w) dP_{\mathbb{D}}(w)  \; \textmd{as} \;  P_{\mathbb{D}}(B)=0. \\
\end{split}
\end{equation*}
and
\begin{equation*}
\begin{split}
P_{\mathbb{D}}(A)  E_{A}(X_{\mathbb{D}}) + P_{\mathbb{D}}(B)  E_{B}(X_{\mathbb{D}})
& = P_{\mathbb{D}}(A)  E_{A}(X_{\mathbb{D}}) \\
& =  \mu_{1}\; e \left[ \int_{A} \frac {X_{\mathbb{D}}(w) dP_{\mathbb{D}}(w)}{\mu_{1}} e + E(X_{\mathbb{D}}) e^{\dagger} \right] \\ 
& = e  \int_{A} X_{\mathbb{D}}(w) dP_{\mathbb{D}}(w).\\
\end{split}
\end{equation*}
Therefore $$P_{\mathbb{D}}(A \cup B) E_{A\cup B}(X_{\mathbb{D}})=  P_{\mathbb{D}}(A)  E_{A}(X_{\mathbb{D}}) + P_{\mathbb{D}}(B)  E_{B}(X_{\mathbb{D}}).$$
 \subsubitem(ii) If $\mu_{1}=0$ and $\mu_{2}>0,$ then $ P_{\mathbb{D}}(A)=0.$
      Now
\begin{equation*}
\begin{split}
P_{\mathbb{D}}(A \cup B) E_{A\cup B}(X_{\mathbb{D}})
& = \lambda_{1} e \left[ \int_{A \cup B} \frac{X_{\mathbb{D}}(w) dP_{\mathbb{D}}(w)}{\lambda_{1}}e + E(X_{\mathbb{D}})e^{\dagger}\right]\\
& =  e  \int_{A \cup B} X_{\mathbb{D}}(w) dP_{\mathbb{D}}(w) \\ 
& = e \left[ \int_{A} X_{\mathbb{D}}(w) dP_{\mathbb{D}}(w)+  \int_{B} X_{\mathbb{D}}(w) dP_{\mathbb{D}}(w)  \right]\\
& = e  \int_{B} X_{\mathbb{D}}(w) dP_{\mathbb{D}}(w)  \; \textmd{as} \;  P_{\mathbb{D}}(A)=0. \\
\end{split}
\end{equation*}
and
\begin{equation*}
\begin{split}
P_{\mathbb{D}}(A)  E_{A}(X_{\mathbb{D}}) + P_{\mathbb{D}}(B)  E_{B}(X_{\mathbb{D}})
& = P_{\mathbb{D}}(B)  E_{B}(X_{\mathbb{D}}) \\
& =  \mu_{2}\; e \left[ \int_{B} \frac {X_{\mathbb{D}}(w) dP_{\mathbb{D}}(w)}{\mu_{2}} e + E(X_{\mathbb{D}}) e^{\dagger} \right] \\ 
& = e  \int_{B} X_{\mathbb{D}}(w) dP_{\mathbb{D}}(w).\\
\end{split}
\end{equation*}
Therefore $$P_{\mathbb{D}}(A \cup B) E_{A\cup B}(X_{\mathbb{D}})=  P_{\mathbb{D}}(A)  E_{A}(X_{\mathbb{D}}) + P_{\mathbb{D}}(B)  E_{B}(X_{\mathbb{D}}).$$
\subsubitem(iii) If $\mu_{1}>0$ and $\mu_{2}>0,$ then $ P_{\mathbb{D}}(A)=\mu_{1}e$ and $ P_{\mathbb{D}}(B)=\mu_{2}e$ 
      Now
\begin{equation*}
\begin{split}
&  P_{\mathbb{D}}(A)  E_{A}(X_{\mathbb{D}}) + P_{\mathbb{D}}(B)  E_{B}(X_{\mathbb{D}}) \\
& = \mu_{1}e \left[\int_{A} \frac{X_{\mathbb{D}}(w) dP_{\mathbb{D}}(w)}{\mu_{1}}e + E(X_{\mathbb{D}})e^{\dagger}\right] + \mu_{2}e \left[\int_{B} \frac{X_{\mathbb{D}}(w) dP_{\mathbb{D}}(w)}{\mu_{2}}e + E(X_{\mathbb{D}})e^{\dagger}\right]\\
& =  e \left [\int_{A} X_{\mathbb{D}}(w) dP_{\mathbb{D}}(w) + \int_{B} X_{\mathbb{D}}(w) dP_{\mathbb{D}}(w)\right]\\
\end{split}
\end{equation*} and
\begin{equation*}
\begin{split}
&  P_{\mathbb{D}}(A \cup B)  E_{A \cup B}(X_{\mathbb{D}}) \\
& = \lambda_{1}e \left[\int_{A \cup B } \frac{X_{\mathbb{D}}(w) dP_{\mathbb{D}}(w)}{\lambda_{1}}e + E(X_{\mathbb{D}})e^{\dagger}\right] \\
& = e \int_{A \cup B } \frac{X_{\mathbb{D}}(w) dP_{\mathbb{D}}(w)}{\lambda_{1}} \\
& =  e \left [\int_{A} X_{\mathbb{D}}(w) dP_{\mathbb{D}}(w) + \int_{B} X_{\mathbb{D}}(w) dP_{\mathbb{D}}(w)\right].\\
\end{split}
\end{equation*}
Therefore $$P_{\mathbb{D}}(A \cup B) E_{A\cup B}(X_{\mathbb{D}})=  P_{\mathbb{D}}(A)  E_{A}(X_{\mathbb{D}}) + P_{\mathbb{D}}(B)  E_{B}(X_{\mathbb{D}}).$$ Hence \eqref{eq1} follows.
\item[(d)] If $ P_{\mathbb{D}}(A \cup B) = \lambda_{2}\; e^{\dagger} $ with $ \lambda_{2}>0,$ the proof follows analogously. 
\end{enumerate}
\end{proof}
The concept of conditional expectation can be extended if the conditioning is not just for one event, but for a countable collection of events, say $ \left\lbrace A_{n}/ n\in \mathbb{N} \right\rbrace.$
Let $(\Omega,\Sigma,P_{\mathbb{D}})$ be a $\mathbb{D}$-probabilistic space and $ \mathbb{P} = \left\lbrace A_{n}/ n\in \mathbb{N} \right\rbrace $ be a partition of $ \Omega $ with positive probability of each member. Then $ \left\lbrace P_{\mathbb{D}}(./A_{n}) / n\in \mathbb{N}\right\rbrace $ is a family of conditional probability functions on $ \Sigma $ and the conditional expectation of $\mathbb{D}$-random variable $ X_{\mathbb{D}},$ with $ E(|X_{\mathbb{D}}|_{k})\prec \infty, $ given a partition $ \mathbb{P}$ is a countably valued function denoted by $ E_{\mathbb{P}}(X_{\mathbb{D}}),$ with values $ \left\lbrace E_{A_{n}}(X_{\mathbb{D}})/ n \in \mathbb{N} \right\rbrace $ given by $$ E_{\mathbb{P}}(X_{\mathbb{D}}) = \sum_{n=1}^{\infty} E_{A_{n}}(X_{\mathbb{D}}) \chi_{A_{n}},$$ where $ \chi_{A}$ is the indicator function of A.\\Let $\mathcal{B}$ be the smallest $\sigma$-algebra containing the partition $\mathbb{P}.$ Then the generators of $\mathcal{B}$ are of the form $ \cup_{k\in J}A_{k}$ for subsets $J\subset \mathbb{N},A_{k} \in\mathbb{P}.$ Since $ E_{\mathbb{P}}(X_{\mathbb{D}})$ is a $ \mathbb{D} $- random variable relative to $ \mathcal{B},$we can integrate it on each generator A of $\mathcal{B}$ to obtain
\begin{equation}\label{eqn3} \int_{A}E_{\mathbb{P}}(X_{\mathbb{D}})dP_{\mathbb{D}} =  \int_{A}X_{\mathbb{D}}\;dP_{\mathbb{D}}.
\end{equation}
Let us consider the different cases that arise.\\
\begin{enumerate}
\item[(a)] If $P_{\mathbb{D}}(A_{n})\notin \mathfrak{S}_{\mathbb{D}},\forall n,$then
\begin{equation*}
\begin{split}
\int_{A}E_{\mathbb{P}}(X_{\mathbb{D}})dP_{\mathbb{D}}&= \sum_{n=1}^{\infty} E_{A_{n}}(X_{\mathbb{D}})   \int_{A}\chi_{A_{n}}dP_{\mathbb{D}}\\
&= \sum_{n\in J} \int_{A_{n}} X_{\mathbb{D}}\;dP_{\mathbb{D}}. \frac{P_{\mathbb{D}}(A\cap A_{n}}{P_{\mathbb{D}}(A_{n})}\\
&= \sum_{n\in J} \int_{A_{n}} X_{\mathbb{D}}\;dP_{\mathbb{D}},\;\text{since}\;  A_{n} \subset A \; \text{for}\;  n \in J.\\
&= \int_{A}X_{\mathbb{D}}\;dP_{\mathbb{D}}\\
\end{split}
\end{equation*}
\item[(b)] If $P_{\mathbb{D}}(A_{n_{0}})\notin \mathfrak{S}_{\mathbb{D}},$ for some $ n_{0} \in \mathbb{N},$ then
$P_{\mathbb{D}}(A_{n})=\lambda_{n}e$ or $P_{\mathbb{D}}(A_{n})=\mu_{n}e^{\dagger},\forall n\neq n_{0}.$\\
For $P_{\mathbb{D}}(A_{n})=\lambda_{n}e,$ for all $n\neq n_{0},$ we have $$E_{A_{n}}(X_{\mathbb{D}})=\int_{A_{n}} \frac{X_{\mathbb{D}}(w)dP_{\mathbb{D}}(w)}{\lambda_{n}} e + E(X_{\mathbb{D}})e^{\dagger}.$$
Therefore,
\begin{equation*}
\begin{split}
& \int_{A}E_{\mathbb{P}}(X_{\mathbb{D}})dP_{\mathbb{D}}\\
&= \sum_{n=1}^{\infty} E_{A_{n}}(X_{\mathbb{D}})   \int_{A}\chi_{A_{n}}dP_{\mathbb{D}}\\
&= \sum_{n=1}^{\infty} E_{A_{n}}(X_{\mathbb{D}})P_{\mathbb{D}}(A \cap A_{n})\\
&= \sum_{n=1}^{\infty} E_{A_{n}}(X_{\mathbb{D}})P_{\mathbb{D}}( A_{n})\\
&= \sum_{n\in J-\left\lbrace n_{0}\right\rbrace} E_{A_{n}}(X_{\mathbb{D}})P_{\mathbb{D}}( A_{n}) + E_{A_{n_{0}}}(X_{\mathbb{D}})P_{\mathbb{D}}( A_{n_{0}})\\
&= \sum_{n\in J-\left\lbrace n_{0}\right\rbrace} \left[\int_{A_{n}} \frac{X_{\mathbb{D}}(w)dP_{\mathbb{D}}(w)}{\lambda_{n}} e + E(X_{\mathbb{D}})e^{\dagger} \right]\lambda_{n}e \\
& ~~~~~~~~~~~ + \int_{A_{n_{0}}} X_{\mathbb{D}}(w)dP_{\mathbb{D}}(w)\\
\end{split}
\end{equation*} 

\begin{equation*}
\begin{split}
&= \sum_{n\in J-\left\lbrace n_{0}\right\rbrace} e \int_{A_{n}} X_{\mathbb{D}}(w)dP_{\mathbb{D}}(w) +  \int_{A_{n_{0}}} X_{\mathbb{D}}(w)dP_{\mathbb{D}}(w)\\
&= \sum_{n\in J-\left\lbrace n_{0}\right\rbrace} \left[ e \int_{A_{n}} X_{1}(w)dP_{1}(w) + e^{\dagger}.0 \right] \\
& ~~~~~~~~~~~ + \left[e \int_{A_{n_{0}}} X_{1}(w)dP_{1}(w) + e^{\dagger} \int_{A_{n_{0}}} X_{2}(w)dP_{2}(w) \right]\\
&= \sum_{n\in J} \left[ e \int_{A_{n}} X_{1}(w)dP_{1}(w) + e^{\dagger} \int_{A_{n}} X_{2}(w)dP_{2}(w)\right]\\
&= \sum_{n\in J} \int_{A_{n}} X_{\mathbb{D}}(w)dP_{\mathbb{D}}(w)\\
&= \int_{A}X_{\mathbb{D}}\;dP_{\mathbb{D}}\\
\end{split}
\end{equation*} 
The proof follows analogously for the other case.
\item[(c)]
If $P_{\mathbb{D}}(A_{n})\in \mathfrak{S}_{\mathbb{D}},$ for all $ n \in \mathbb{N},$ then
$P_{\mathbb{D}}(A_{n})=\lambda_{n}e$ or $P_{\mathbb{D}}(A_{n})=\mu_{n}e^{\dagger},\forall n.$\\
For $P_{\mathbb{D}}(A_{n})=\lambda_{n}e,$ for all n, we have $$E_{A_{n}}(X_{\mathbb{D}})=\int_{A_{n}} \frac{X_{\mathbb{D}}(w)dP_{\mathbb{D}}(w)}{\lambda_{n}} e + E(X_{\mathbb{D}})e^{\dagger}$$
Therefore,
\begin{equation*}
\begin{split}
\int_{A}E_{\mathbb{P}}(X_{\mathbb{D}})dP_{\mathbb{D}}&= \sum_{n=1}^{\infty} E_{A_{n}}(X_{\mathbb{D}})   \int_{A}\chi_{A_{n}}dP_{\mathbb{D}}\\
&= \sum_{n=1}^{\infty} E_{A_{n}}(X_{\mathbb{D}})   P_{\mathbb{D}}(A_{n})\\
&= \sum_{n=1}^{\infty}\left[\int_{A_{n}} \frac{X_{\mathbb{D}}(w)dP_{\mathbb{D}}(w)}{\lambda_{n}} e + E(X_{\mathbb{D}})e^{\dagger}\right]\lambda_{n}e\\
&=  \sum_{n=1}^{\infty}e \int_{A_{n}} X_{\mathbb{D}}(w)dP_{\mathbb{D}}(w)\\
&=  \sum_{n=1}^{\infty}\left[e \int_{A_{n}} X_{1}(w)dP_{1}(w) + 0\right]\\
&=  \sum_{n=1}^{\infty} \left[e \int_{A_{n}} X_{1}(w)dP_{1}(w) + e^{\dagger} \int_{A_{n}} X_{2}(w)dP_{2}(w)\right]\\
&=  \int_{A} X_{\mathbb{D}}\;dP_{\mathbb{D}}.\\
\end{split}
\end{equation*}  
The proof is analogous for the other case.\\
Hence $\int_{A}E_{\mathbb{P}}(X_{\mathbb{D}})dP_{\mathbb{D}} =  \int_{A}X_{\mathbb{D}}\;dP_{\mathbb{D}}.$
\end{enumerate}
\begin{theorem}
Let $ X_{\mathbb{D}}$ be a $\mathbb{D}$-random variable on a $\mathbb{D}$-probabilistic space $(\Omega,\Sigma,P_{\mathbb{D}})$ and $\mathbb{P} \subset \Sigma $ be a partition of $\Omega$ (of positive probability of each member).If $\mathbb{P}=\left\lbrace A_{\imath}/\imath \in \mathbb{N} \right\rbrace, $ then the existence of $ E_{\mathbb{P}}(X_{\mathbb{D}})$ implies that $ E_{\mathcal{A}}(X_{\mathbb{D}})$ exists for each sub partition $\mathcal{A}= \left\lbrace A_{\imath_{n}}/n \in \mathbb{N} \right\rbrace \subset \mathbb{P} $ and one has \\
\begin{equation}\label{eqn4}
 E_{A}(X_{\mathbb{D}}) P_{\mathbb{D}}(A) =   \int_{A}E_{\mathbb{P}}(X_{\mathbb{D}})dP_{\mathbb{D}}= P_{\mathbb{D}}(A) \sum_{n=1}^{\infty} E_{A_{\imath_{n}}}(X_{\mathbb{D}})P_{\mathbb{D}}(A_{\imath_{n}}/A),
 \end{equation} where $ A= \cup_{n=1}^{\infty}A_{\imath_{n}}.$ 
\end{theorem}
\begin{proof}
we have $ E_{\mathbb{P}}(X_{\mathbb{D}})(w) = E_{A_{n}}(X_{\mathbb{D}})(w),w\in A_{n},n\in \mathbb{N}. $
Therefore existence of $ E_{\mathbb{P}}(X_{\mathbb{D}})$ implies that of $ E_{A_{n}}(X_{\mathbb{D}}), n \in \mathbb{N}$  and hence of $ E_{A}(X_{\mathbb{D}})$ or $ E_{\mathcal{A}}(X_{\mathbb{D}}).$Also by \ref{eqn3}, we have $ \int_{A}E_{\mathbb{P}}(X_{\mathbb{D}})dP_{\mathbb{D}}= \int_{A}X_{\mathbb{D}}\;dP_{\mathbb{D}} = P_{\mathbb{D}}(A)E_{A}(X_{\mathbb{D}}).$ There arise three cases.\\
\begin{enumerate}
\item[(a)] If $ P_{\mathbb{D}}(A) \notin \mathfrak{S}_{\mathbb{D}},$ then $ P_{\mathbb{D}}(A_{\imath_{n}}) \notin \mathfrak{S}_{\mathbb{D}},\forall \; n \; \text{or} \; P_{\mathbb{D}}(A_{\imath_{n_{0}}}) \notin \mathfrak{S}_{\mathbb{D}}$ for some $ n_{0}.$ 
\subitem(i) $ P_{\mathbb{D}}(A_{\imath_{n}}) \notin \mathfrak{S}_{\mathbb{D}},\forall \; n.$\\
 In this case,\\
\begin{equation*}
\begin{split}
\int_{A}X_{\mathbb{D}}\;dP_{\mathbb{D}}
&= \sum_{n=1}^{\infty} \int_{A_{\imath_{n}}} X_{\mathbb{D}} dP_{\mathbb{D}}\\
&= \sum_{n=1}^{\infty} E_{A_{\imath_{n}}}(X_{\mathbb{D}})P_{\mathbb{D}}(A_{\imath_{n}})\\ 
&= \sum_{n=1}^{\infty} E_{A_{\imath_{n}}}(X_{\mathbb{D}})P_{\mathbb{D}}(A_{\imath_{n}}/A)P_{\mathbb{D}}(A),\\
\end{split}
\end{equation*}
hence \ref{eqn4} holds.
\subitem(ii) If $ P_{\mathbb{D}}(A_{\imath_{n_{0}}}) \notin \mathfrak{S}_{\mathbb{D}},$ for some $ n_{0} \in \mathbb{N}, $ then $ P_{\mathbb{D}}(A_{\imath_{n}})=\lambda_{n} e$ or $ P_{\mathbb{D}}(A_{\imath_{n}})=\mu_{n} e^{\dagger}, \forall  n \neq n_{0}.$ 
If $ P_{\mathbb{D}}(A_{\imath_{n}})=\lambda_{n} e,$ then $ P_{1}(A_{\imath_{n}})=\lambda_{n} e$ and $ P_{2}(A_{\imath_{n}})=0,\forall\; n\neq n_{0}.$ In this case,
\begin{equation*}
\begin{split}
P_{\mathbb{D}}(A) \sum_{n=1}^{\infty} E_{A_{\imath_{n}}}(X_{\mathbb{D}})P_{\mathbb{D}}(A_{\imath_{n}}/A)
& = P_{\mathbb{D}}(A) \sum_{n=1}^{\infty} E_{A_{\imath_{n}}}(X_{\mathbb{D}})\frac {P_{\mathbb{D}}(A_{\imath_{n}} \cap A)}{P_{\mathbb{D}}(A)}\\
& =P_{\mathbb{D}}(A) \sum_{n=1}^{\infty} E_{A_{\imath_{n}}}(X_{\mathbb{D}})\frac {P_{\mathbb{D}}(A_{\imath_{n}})} {P_{\mathbb{D}}(A)}\\
& = \sum_{n=1}^{\infty} E_{A_{\imath_{n}}}(X_{\mathbb{D}})P_{\mathbb{D}}(A_{\imath_{n}})\\
& = \sum_{n=1,n\neq n_{0}}^{\infty} E_{A_{\imath_{n}}}(X_{\mathbb{D}})P_{\mathbb{D}}(A_{\imath_{n}} ) + E_{A_{\imath_{n_{0}}}}(X_{\mathbb{D}})P_{\mathbb{D}}(A_{\imath_{n_{0}}}) \\
& = \sum_{n=1,n\neq n_{0}}^{\infty}\left[\int_{A_{\imath_{n}}} \frac {X_{\mathbb{D}} dP_{\mathbb{D}}}{\lambda_{n}} e + E_{X_{\mathbb{D}}} e^{\dagger}\right]  + \int_{A_{\imath_{n_{0}}}} X_{\mathbb{D}} dP_{\mathbb{D}} \\
&= \left[\sum_{n=1,n\neq n_{0}}^{\infty} e \int_{A_{\imath_{n}}} X_{\mathbb{D}} dP_{\mathbb{D}}+ \int_{A_{\imath_{n_{0}}}} X_{\mathbb{D}} dP_{\mathbb{D}}  \right]\\
&= \left[\sum_{n=1,n\neq n_{0}}^{\infty}  \int_{A_{\imath_{n}}} X_{\mathbb{D}} dP_{\mathbb{D}}\right]\\
&= \int_{A} X_{\mathbb{D}} dP_{\mathbb{D}},\\
\end{split}
\end{equation*}
because $ \int_{A_{\imath_{n}}} X_{\mathbb{D}} dP_{\mathbb{D}}= e \int_{A_{\imath_{n}}} X_{1} dP_{1} +  e^{\dagger} \int_{A_{\imath_{n}}} X_{2} dP_{2}=e \int_{A_{\imath_{n}}} X_{1} dP_{1}=e \int_{A_{\imath_{n}}} X_{\mathbb{D}} dP_{\mathbb{D}}$ as $ P_{2}(A_{\imath_{n}})=0, \forall n\neq n_{0}.$ \\
If $ P_{2}(A_{\imath_{n}})=\mu_{n} e^{\dagger}$ and $ P_{1}(A_{\imath_{n}})=0, \forall n\neq n_{0},$ the proof is analogous.
\item[(b)] If $ P_{\mathbb{D}}(A)=\lambda e,\lambda>0,$ then $ P_{\mathbb{D}}(A_{\imath_{n}})=\lambda_{n}e,\forall n\in \mathbb{N},$ where $ \lambda = \sum_{n=1}^{\infty}\lambda_{n}.$\\Therefore
$$ \int_{A} X_{\mathbb{D}} dP_{\mathbb{D}}= \sum_{n=1}^{\infty}  \int_{A_{\imath_{n}}} X_{\mathbb{D}} dP_{\mathbb{D}}= e \sum_{n=1}^{\infty}  \int_{A_{\imath_{n}}} X_{\mathbb{D}} dP_{\mathbb{D}} = e \int_{A} X_{\mathbb{D}} dP_{\mathbb{D}} $$ as $P_{2}(A_{\imath_{n}})=0,\forall n\in \mathbb{N}.$\\and
\begin{equation*}
\begin{split}
P_{\mathbb{D}}(A) \sum_{n=1}^{\infty} E_{A_{\imath_{n}}}(X_{\mathbb{D}})P_{\mathbb{D}}(A_{\imath_{n}}/A)\\
&= P_{\mathbb{D}}(A) \sum_{n=1}^{\infty} E_{A_{\imath_{n}}}(X_{\mathbb{D}})\frac {P_{\mathbb{D}}(A_{\imath_{n}} \cap A)} {P_{\mathbb{D}}(A)}\\
&= \sum_{n=1}^{\infty} E_{A_{\imath_{n}}}(X_{\mathbb{D}})P_{\mathbb{D}}(A_{\imath_{n}})\\
& = \sum_{n=1}^{\infty}\left[\int_{A_{\imath_{n}}} \frac {X_{\mathbb{D}} dP_{\mathbb{D}}}{\lambda_{n}} e + E({X_{\mathbb{D}}}) e^{\dagger}\right] \lambda_{n}e\\
&= \sum_{n=1}^{\infty}e \int_{A_{\imath_{n}}} X_{\mathbb{D}} dP_{\mathbb{D}}\\
&= e \sum_{n=1}^{\infty} \int_{A_{\imath_{n}}} X_{\mathbb{D}} dP_{\mathbb{D}}\\
&= \int_{A} X_{\mathbb{D}} dP_{\mathbb{D}}.\\
\end{split}
\end{equation*}
On the other hand $$ E_{A}(X_{\mathbb{D}})P_{\mathbb{D}}(A) =\left[ \int_{A} \frac{X_{\mathbb{D}} dP_{\mathbb{D}}}{\lambda} e + E(X_{\mathbb{D}}) e^{\dagger}\right] \lambda e =e \int _{A}X_{\mathbb{D}} dP_{\mathbb{D}}= \int _{A}X_{\mathbb{D}} dP_{\mathbb{D}}.$$ 
\item[(c)] The proof is analogous for the case when $ P_{\mathbb{D}}(A)=\mu e^{\dagger}$ with $\mu>0.$
\end{enumerate}
Hence $\ref{eqn4}$ holds for all the cases.
\end{proof}
Formula \ref{eqn4} is useful in calculating $ E_{A}(X_{\mathbb{D}})$ when individual $ E_{A_{n}}(X_{\mathbb{D}})$'s are known.  
The whole theory of probability can be generalized for $\mathbb{D}$-random variables on the similar lines.
\end{section}
\bibliographystyle{amsplain}
 
\noindent Romesh Kumar, \textit{Department of Mathematics, University of Jammu, Jammu, J\&K - 180 006, India.}\\
E-mail :\textit{ romesh\_jammu@yahoo.com}\\
\noindent Kailash Sharma, \textit{Department of Mathematics, University of Jammu, Jammu,  J\&K - 180 006, India.}\\
E-mail :\textit{ kailash.maths@gmail.com}\\
\end{document}